\documentclass[11pt,a4paper,reqno]{amsart}%
\usepackage{amssymb}
\usepackage{amsmath}
\usepackage{amsfonts}
\usepackage{color}
\usepackage[usenames,dvipsnames]{xcolor}
\usepackage{bbm}
\usepackage{graphicx}
\usepackage{rotating}
\usepackage{hyperref}
\usepackage{mathrsfs}
\usepackage[bb=boondox]{mathalfa}
\usepackage[all]{xy}%
\providecommand{\U}[1]{\protect\rule{.1in}{.1in}}
\newtheorem{theorem}{Theorem}

\newtheorem{corollary}[theorem]{Corollary}

\newtheorem{definition}[theorem]{Definition}
\newtheorem{example}[theorem]{Example}

\newtheorem{lemma}[theorem]{Lemma}

\newtheorem{proposition}[theorem]{Proposition}
\newtheorem{remark}[theorem]{Remark}

\thanks{}
\begin{document}
\newcommand{\R}{{\mathbb R}}
\newcommand{\C}{{\mathbb C}} 
\newcommand{\T}{{\mathbb T}}
\newcommand{\D}{{\mathbb D}}
\renewcommand{\P}{\mathbb P}

\newcommand{\Aa}{{\mathcal A}}
\newcommand{\Ii}{{\mathcal I}}
\newcommand{\Jj}{{\mathcal J}}
\newcommand{\Nn}{{\mathcal N}}
\newcommand{\Ll}{{\mathcal L}}
\newcommand{\Tt}{{\mathcal T}}
\newcommand{\Gg}{{\mathcal G}}
\newcommand{\Dd}{{\mathcal D}}
\newcommand{\Cc}{{\mathcal C}}
\newcommand{\Oo}{{\mathcal O}}

\newcommand{\pr}{\operatorname{pr}}
\newcommand{\bla}{\langle \! \langle}
\newcommand{\bra}{\rangle \! \rangle}
\newcommand{\blq}{[ \! [}
\newcommand{\brq}{] \! ]}

\title{Holomorphic Jacobi Manifolds}

\author{Luca Vitagliano}
\address{DipMat, Universit\`a degli Studi di Salerno, via Giovanni Paolo II n${}^\circ$ 123, 84084 Fisciano (SA) Italy.}
\email{lvitagliano@unisa.it}

\author{A\"issa Wade}
\address{Department of Mathematics, Penn State University, University Park, State College, PA 16802, USA.}
\email{wade@math.psu.edu}

\begin{abstract}
In this paper, we develop holomorphic Jacobi structures. Holomorphic Jacobi manifolds are in one-to-one correspondence with certain homogeneous holomorphic Poisson manifolds. Furthermore, holomorphic Poisson manifolds can be looked at as special cases of holomorphic Jacobi manifolds.  We show that holomorphic Jacobi structures yield a much richer framework than that of holomorphic Poisson structures. We also discuss the relationship between  holomorphic Jacobi structures, generalized contact bundles and Jacobi-Nijenhuis structures.
% \color{red}Finally, we study the integration problem for holomorphic Jacobi Nijenhuis manifold.  

\end{abstract}
\maketitle

%\tableofcontents


%%%%%%%%%%%%%%%%% HOLOMORPHIC POISSON %%%%%%%%%%%%%%%%%%%%%
\section{Introduction}

In this paper, we study  %special cases of generalized contact bundles (see Theorem \ref{theor:hJ_JN_gc} and Reference \cite{VW2016}). Namely,  these are called 
holomorphic Jacobi manifolds.  By a holomorphic Jacobi manifold, we mean a complex manifold $X$ equipped with
a holomorphic line bundle $L \to X$ together with a holomorphic  bi-derivation $J$ of $L$ so that $(L, J)$ is a Jacobi structure (Definition \ref{def:main}, see also \cite{Kirillov1976, Lichn1978, Marle1991, CS2015} for more details about real Jacobi manifolds). Equivalently, $J$ is completely determined by 
the data of a Jacobi $\Oo_X$-module on the sheaf $\Gamma_L$ of holomorphic sections of $L$, where $\Oo_X$ is the standard sheaf of holomorphic functions on $X$. In other words,
 for each open set $U \subset X$, $\Gamma_L (U)$ is an $\mathcal O_X(U)$-module endowed with a Lie bracket
$\{-,-\} : \Gamma_L (U) \times \Gamma_L (U) \to \Gamma_L (U)$ and a Lie algebra homomorphism $R_U: \Gamma_L (U) \to \operatorname{Der} \mathcal O_X(U)$ satisfying:
\[
 \{ \lambda_1, f \lambda_2 \} = \big( R_U({\lambda_1}) f \big) \lambda_2 + f \{ \lambda_1, \lambda_2 \},
\]
for all $\lambda_1, \lambda_2 \in \Gamma_L(U)$ and for all $f \in \mathcal O_X(U)$.  In addition, these Lie brackets and homomorphisms $R_U$ are all required to be compatible with restrictions.
 Non-degenerate holomorphic Jacobi structures are holomorphic contact structures which naturally appear in many contexts (see, e.g.~\cite{L1995, S1982, NT1987}).
 
Recently,  holomorphic Poisson structures were intensively studied due to their close relationship with generalized complex geometry. 
In fact, it was proved in \cite{Bailey} that, locally, any generalized complex manifold is the product of a symplectic manifold by a holomorphic Poisson manifold. Moreover, deformations of holomorphic Poisson manifolds were investigated by several mathematicians \cite{H2012, G2010, FM2012} while their integration problem was considered in \cite {LSX2008}. 
  A quite natural question is whether all these results on holomorphic Poisson structures can be extended to the setting of Jacobi geometry. But this question still remained open despite the close relationship between  Poisson manifolds and Jacobi manifolds.
  
  Recall that Jacobi structures on manifolds were  introduced independently by Lichnerowicz \cite{Lichn1978} and Kirillov \cite{Kirillov1976} who used different but basically equivalent definitions. Adopting Lichnerowicz's  perspective, one easily see that the category of Poisson manifolds is a subcategory of the category of Jacobi manifolds. But Jacobi manifolds can also be regarded as \emph{homogeneous Poisson manifolds}. Here, by a \emph{homogeneous Poisson manifold} \cite{DLM1991}, we mean a manifold $M$ equipped with a Poisson bi-vector $\pi$ together with a vector field $\eta$ such that $[\eta, \pi]^{SN} = \mathcal L_\eta \pi = -\pi$, where $[-,-]^{SN}$ is the Schouten-Nijenhuis bracket.  Nonetheless, almost nothing is known about holomorphic Jacobi structures. The present paper is the first in a series aiming at filling this gap. 
   In both the real and complex cases, we will adopt the general line bundle  approach  to Jacobi geometry  (see \cite{Marle1991, V2015, VW2016}). Of course, this approach includes standard Jacobi structures as considered by  Lichnerowicz \cite{Lichn1978}. The latter are defined as pairs $(\pi, E)$ of tensors consisting of a bivector field $\pi$ and a vector field $E$ satisfying $[\pi, \pi]^{SN}=2 E \wedge \pi $ and $[\pi, E]^{SN}=0$, and correspond to the case of the trivial line bundle $M \times \mathbb R \to M$ \cite{Marle1991, V2015, VW2016}. Passing from trivial to general line bundles has relevant conceptual advantages, but, in the real case, it represents only a minor improvement from a practical point of view since,  in the real setting, a line bundle is always trivial up to taking a double cover of its base manifold. Of course, this is far from being the case in the holomorphic setting. Hence, in the holomorphic case, taking the general line bundle point of view is of a particular importance. In fact, this leads to new interesting features (see below).

 In this paper, we want to understand holomorphic Jacobi manifolds from a real 
 differential geometric point of view as  carried out  for holomorphic Poisson 
 manifolds in \cite{LSX2008}. In particular, our focus is always on real differential 
 geometry. For instance we describe the gauge algebroid and the first jet bundle of a 
 holomorphic vector bundle in real geometric terms. One of our main ingredients is 
 the so called homogenization scheme (see \cite{VW2019}). More precisely, there is 
 an equivalence between the categories of real line bundles  and principal $\mathbb
  R^\times$-bundles, where $\mathbb R^\times$ denotes the multiplicative group of 
  non-zero reals. Indeed, given a real line bundle $\ell \to M$,  its slit dual bundle  
  ${\widetilde M}=\ell^\ast \smallsetminus \{0 \}$  is a principal bundle over $M$ with
   structure group the multiplicative group $\mathbb R^\times$. We call $\widetilde M$ the \emph{homogenization} of $\ell$. Given a real line bundle $\ell \to M$ together with a Jacobi structure $J$, there is a homogeneous Poisson structure $\pi$ on $\widetilde M$. The pairs $(\ell, J)$ and $(\widetilde M, \pi)$ encode the same mathematical information. A similar construction works in the holomorphic setting. However, when trying to describe holomorphic Jacobi manifolds from a real point of view, new complications pop up. Namely, there are subtle differences between holomorphic Poisson structures and holomorphic Jacobi structures due, in part, to the fact the a holomorphic line bundle $L$ over a complex manifold $X=(M, j)$ is not a rank 1 but it's a rank 2 real vector bundle over $M$. So, a direct homogenization scheme won't  work when one tries to generalize results from \cite{LSX2008} to holomorphic Jacobi structures. In Section \ref{sec:HJ_JN_GC}, we overcame this difficulty by noticing that, in the complex case, one can perform the following  two-step procedure. Firstly, one passes  from a holomorphic line bundle  $L \to X$ to a real $U(1)$-principal bundle $\widehat M \to M$ equipped with a canonical real line bundle $\widehat L$ by taking $\widehat M = \mathbb R \mathbb P (L^*)$, the real projectivization of the dual $L^\ast$ (which is obviously of odd dimension), and the dual  $\widehat L \to \widehat M$ of the tautological bundle. 
 Secondly, one checks that the (real) homogenization of  $\widehat L \to \widehat M$ agrees with the complex homogenization $\widetilde X= L^\ast \smallsetminus  \{0\}$. Of course, one should keep track of the complex structure  along all this procedure.

  The paper is divided into three parts. In Section \ref{Sec:Basics}, we explore basic definitions and results  needed for a better understanding of holomorphic Jacobi structures.
  Section \ref{Sec:HolomJacobi} defines holomorphic Jacobi structures and unravels their properties as well as their relationship with Jacobi-Nijenhuis \cite{MMP1999}, generalized contact structures \cite{VW2016}  and analogous \emph{homogeneous} structures in the realm of Poisson geometry. In sum, we explain the following diagram:
  \[  
\xymatrix{ \text{Jacobi-Nijenhuis} \ar@{<=>}[d] & \text{holomorphic Jacobi} \ar@{=>}[r] \ar@{=>}[l] \ar@{<=>}[d] & \text{generalized contact} \ar@{<=>}[d] \\
\txt{homogeneous \\ Poisson-Nijenhuis} \ 
& \txt{homogeneous \\ holomorphic Poisson}\  \ar@{=>}[r] \ar@{=>}[l] & \txt{homogeneous \\ generalized complex}\ 
}
\]

  Finally, Section \ref{Sec:LieAlgJacobi} studies the Lie algebroid of a holomorphic Jacobi manifold.
    
  In this paper, undecorated tensor products and homomorphisms are over $\C$ or complex valued smooth functions, unless otherwise stated. Moreover, if $E$ is a (real) vector bundle, we denote by $E^\C := E \otimes_\R \C$ its \emph{complexification}.
  
\section{Basics definitions and results}\label{Sec:Basics}

\subsection{Holomorphic Poisson, Poisson-Nijenhuis and generalized complex manifolds}

In this section we recall the basic definitions and results from \cite{LSX2008}. Let $M$ be a (real) manifold $M$ equipped with a complex structure $j$. 

\begin{definition}
 A \emph{holomorphic Poisson structure}  on the complex manifold $X = (M,j)$ is a \emph{holomorphic Poisson bi-vector}  $\Pi$, that is a bi-vector $\Pi \in \Gamma (\wedge^2 T^{1,0}M)$  satisfying:
   $$\overline{\partial} \Pi = 0 \quad \quad and \quad \quad [ \Pi, \Pi]^{SN} = 0, $$ where $[-,-]^{SN}$ is the \emph{Schouten-Nijenhuis bracket} of complex multivectors. A \emph{holomorphic Poisson manifold} is a complex manifold equipped with a holomorphic Poisson structure.
\end{definition}

\begin{remark}
A holomorphic Poisson manifold is nothing but  a complex manifold $X$ equipped with the structure of a sheaf of Poisson algebra on its sheaf $\Oo_X$ of holomorphic sections. Given a holomorphic Poisson manifold $(X, \Pi)$, the Poisson bracket of two holomorphic functions $f,g$ on $X$, is given by $\{f,g \} = \Pi (df, dg) = \Pi (\partial f, \partial g)$.
\end{remark}

Characterizations of holomorphic Poisson structures in terms of  Poisson-Nijenhuis structures and  generalized complex structures of special type were given in \cite{LSX2008} (see Theorem \ref{theor:hP_PN_gc} below).
Before reviewing these characterizations, we will recall the basic definitions.

First of all,  every bi-vector $\pi$ on a manifold $M$ determines a skew-symmetric bracket $[-,-]_\pi$ on $1$-forms given by
\begin{equation}\label{eq:[-,-]_pi}
[\rho, \tau]_\pi := \Ll_{\pi^\sharp \rho} \tau - \Ll_{\pi^\sharp \tau} \rho -d \pi (\rho, \tau),
\end{equation}
for all $\rho, \tau \in \Omega^1 (M)$, where $\pi^\sharp : T^\ast M \to TM$ consists in ``raising an index via $\pi$'', i.e.~$\pi^\sharp \rho := \pi (\rho, -)$. A direct computation shows that $\pi$ is a Poisson bi-vector if and only if $[-,-]_\pi$ is a Lie bracket. In this case, $[-,-]_\pi$ is the Lie bracket on sections of the \emph{cotangent algebroid} $(T^\ast M)_\pi$ of the Poisson manifold $(M, \pi)$. Now, let $\phi : TM \to TM$ be a $(1,1)$-tensor, and let $\phi^\ast : T^\ast M \to T^\ast M$ be its transpose. If $\pi^\sharp \circ \phi^\ast = \phi \circ \pi^\sharp$, then 
\begin{equation}\label{Eq:pi_phi}
\pi_\phi := \pi (\phi^\ast -,-)
\end{equation}
 is a well-defined bi-vector such that $\pi^\sharp_\phi = \pi^\sharp \circ \phi^\ast$.

Let $\pi$ be a Poisson bi-vector, and let $\phi$ be a $(1,1)$ tensor on $M$. We say that $\phi$ is \emph{compatible} with $\pi$ if
\begin{equation}
\pi^\sharp \circ \phi^\ast = \phi \circ \pi^\sharp,
\end{equation}
hence $\pi_\phi$ is well-defined by (\ref{Eq:pi_phi}), and 
\begin{equation}
\phi^\ast [\rho, \tau]_\pi = [\phi^\ast \rho, \tau]_\pi + [\rho, \phi^\ast \tau]_\pi - [\rho, \tau]_{\pi_\phi},
\end{equation}
for all $\rho, \tau \in \Omega^1 (M)$.

\begin{definition}
A \emph{Poisson-Nijenhuis manifold} is a manifold $M$ equipped with a \emph{Poisson-Nijenhuis structure}, i.e.~a pair $(\pi, \phi)$, where $\pi$ is a Poisson bi-vector, and $\phi$ is a compatible $(1,1)$ tensor whose \emph{Nijenhuis torsion} $\mathcal N_\phi : \wedge^2 TM \to TM$,   defined by:
\[
\mathcal N_\phi (\xi,\zeta) :=  [\phi (\xi), \phi (\zeta)] + \phi^2 [\xi,\zeta] - \phi [\phi (\xi), \zeta] - \phi [\xi, \phi (\zeta)],
\]
 for all $\xi,\zeta \in \mathfrak X (M)$, vanishes identically.
\end{definition}

\begin{proposition}
Let $(\pi, \phi)$ be a Poisson-Nijenhuis structure. Then $(\pi, \pi_\phi)$ is a \emph{bi-Hamiltonian} structure, i.e.~$\pi$, $\pi_\phi$ and $\pi + \pi_\phi$ are all Poisson bi-vectors.
\end{proposition}

We now recall the definition of a generalized complex manifold \cite{H2003, G2011}. Let $M$ be a manifold. Denote by $\mathbb R_M := M \times \mathbb R \to M$ the trivial line bundle. The \emph{generalized tangent bundle} $\T M := TM \oplus T^\ast M$ is canonically equipped with the following structures:
\begin{itemize}
\item the projection $\operatorname{pr}_T : \mathbb TM \to TM$, 
\item the symmetric bilinear form $\langle \hspace{-2.7pt} \langle -,- \rangle \hspace{-2.7pt} \rangle : \mathbb TM \otimes \mathbb T M \to \mathbb R_M$:
\[
\langle \hspace{-2.7pt} \langle (\xi, \rho), (\zeta, \tau) \rangle \hspace{-2.7pt} \rangle := \tau (\xi) + \rho (\zeta),
\]
\item the \emph{Dorfman bracket} $[\![ -,-]\!] : \Gamma (\mathbb TM) \times \Gamma (\mathbb TM) \to \Gamma (\mathbb TM)$:
\[
[\![  (\xi, \rho), (\zeta, \tau)]\!] := ([\xi,\zeta], \Ll_\xi \tau - \Ll_\zeta \rho + d \rho (\zeta)),
\]
\end{itemize}
for all $\xi, \zeta \in \mathfrak X (M)$, $\rho, \tau \in \Omega^1 (M)$. With the above three structures $\T M$ is a \emph{Courant algebroid}.

\begin{definition}
A \emph{generalized complex manifold} is a manifold $M$ equipped with a \emph{generalized complex structure}, i.e.~a vector bundle endomorphism $\Jj : \T M \to \T M$ such that
\begin{itemize}
\item $\Jj$ is \emph{almost complex}, i.e.~$\Jj^2 = - \mathbb 1$,
\item $\Jj$ is \emph{skew-symmetric}, i.e.
\[
\langle \hspace{-2.7pt} \langle \Jj \alpha, \beta \rangle \hspace{-2.7pt} \rangle + \langle \hspace{-2.7pt} \langle \alpha, \Jj \beta \rangle \hspace{-2.7pt} \rangle = 0, \quad \text{for all }\alpha, \beta \in \Gamma (\T M),
\]
\item $\Jj$ is \emph{integrable}, i.e.
\[
[\![ \mathcal J \alpha, \mathcal J \beta]\!] - [\![ \alpha, \beta]\!] - \mathcal J [\![ \mathcal J \alpha, \beta]\!] + \mathcal J [\![ \alpha, \mathcal J \beta ]\!] = 0, \quad \text{for all }\alpha, \beta \in \Gamma (\T M).
\]
\end{itemize}
\end{definition} 
Let $(M, \Jj)$ be a generalized complex manifold. Using the direct sum decomposition $\T M = TM \oplus T^\ast M$, and the definition, one can see that 
\[
\Jj = \left(
\begin{array}{cc}
\phi & \pi^\sharp \\
\omega_\flat & - \phi^\ast
\end{array}
\right)
\]
where $\pi$ is a Poisson bi-vector, $\phi : TM \to TM$ is an endomorphism compatible with $\pi$, and $\omega$ is a $2$-form, with associated vector bundle morphism $\omega_\flat : TM \to T^\ast M$, satisfying additional compatibility conditions \cite{C2011,SX2007}. In particular, when $\omega = 0$, then $\phi$ is a complex structure, and $(\pi, \phi)$ is a Poisson-Nijenhuis structure. More precisely we have the following

\begin{theorem}[Characterization of holomorphic Poisson structures \cite{LSX2008}] \label{theor:hP_PN_gc}
Let $X = (M,j)$ be a complex manifold, and let $\Pi \in \Gamma (\wedge^2 (TM)^\C)$ be a complex bi-vector on $M$. Denote by $\pi',\pi$ the real and the imaginary part of $\Pi$ respectively: $\Pi = \pi' + i \pi$, $\pi',\pi \in \Gamma (\wedge^2 TM)$. The following conditions are equivalent
\begin{enumerate}
\item $\Pi$ is a holomorphic Poisson structure on $X$,
\item $(\pi, j)$ is a Poisson-Nijenhuis structure on $M$, and $\pi' = \pi_j$ (see Equation \ref{Eq:pi_phi}),  
\item $\Jj := \left(
\begin{smallmatrix}
j & \pi^\sharp \\
0 & - j^\ast
\end{smallmatrix}
\right)$ is a generalized complex structure on $M$, and $\pi' = \pi_j$.
\end{enumerate}
\end{theorem}

%%%%%%%%%%%%%%%%%% HOMOGENOUS HOLOMORPHIC POISSON %%%%%%%%%%%%%%%%%%%%

\subsection{ Homogeneous holomorphic Poisson, homogeneous Poisson-Nijenhuis and homogeneous generalized complex manifolds}\label{section:hhP}
In this section we consider generalized complex manifolds equipped with an additional compatible structure which we call a \emph{homogeneity vector field}. We will also discuss the relationship between holomorphic Poisson, Poisson-Nijenhuis, and generalized complex manifolds in presence of a homogeneity vector field.

\begin{definition}
A \emph{ homogeneous holomorphic Poisson manifold} is a  complex manifold $X = (M, j)$  equipped with a holomorphic Poisson structure $\Pi$ together with a holomorphic vector field $H$ such that $[H, \Pi]^{SN} = \Ll_H \Pi = - \Pi $. The pair $(\Pi, H)$ is called a \emph{ homogeneous holomorphic Poisson structure on $X$}, and we say that $\Pi$ is \emph{homogeneous} with respect to $H$.
\end{definition}

\begin{example}\label{ex:complex_Lie_algebra}
Let $\mathfrak g$ be a complex Lie algebra. Its complex dual $\mathfrak g^\ast$ is equipped with the holomorphic Lie-Poisson structure $\Pi$. By linearity, the holomorphic Euler vector field $H$ on $\mathfrak g^\ast$ is a homogeneity vector field for $\Pi$. Hence, $(\mathfrak g^\ast, \Pi, H)$ is a  homogeneous holomorphic Poisson manifold.
\end{example}

\begin{example}\label{ex:holomorphic_cotangent_bundle}
Let $X = (M, j)$ be a complex manifold, with local holomorphic coordinates $(z^i)$. The cotangent bundle $T^\ast X$ is coordinatized by the $z^i$'s and their conjugated momenta $p_i$. There is a canonical holomorphic symplectic structure $\Omega$ on $T^\ast X$ locally given by $\Omega = dp_i \wedge dz^i$. The associated Poisson structure $\Pi$ is locally given by $\Pi = \frac{\partial}{\partial z^i} \wedge \frac{\partial}{\partial p_i}$. The holomorphic Euler vector field $H$ on $T^\ast X$ is locally given by $H = p_i \frac{\partial}{\partial p_i}$ and it is a homogeneity vector field for $\Pi$. Hence, $(T^\ast X, \Pi, H)$ is a  homogeneous holomorphic Poisson manifold. More generally, let $(X, \Omega)$ be a holomorphic symplectic manifold, with associated Poisson structure $\Pi = \Omega^{-1}$. Additionally, let $H$ be a \emph{homogeneity vector field} for $\Omega$, i.e.~$H$ is a holomorphic vector field on $X$ such that $\mathcal L_H \Omega = \Omega$. Then $H$ is clearly a homogeneity vector field for $\Pi$, hence $(X, \Pi, H)$ is a  homogeneous holomorphic Poisson manifold.
\end{example}

\begin{example}\label{example:holomorphic_dual}
The present example encompasses Examples \ref{ex:complex_Lie_algebra} and \ref{ex:holomorphic_cotangent_bundle} as special cases. Let $A \to X$ be a holomorphic Lie algebroid (see Definition \ref{Def:HolomLieAlg} below, and reference \cite{LSX2008} for more details about holomorphic Lie algebroids). Its complex dual $A^\ast \to X$ is equipped with the holomorphic Lie-Poisson structure $\Pi$. By linearity, the holomorphic Euler vector field $H$ on $A^\ast$ is a homogeneity vector field for $\Pi$. Hence, $(A^\ast, \Pi, H)$ is a  homogeneous holomorphic Poisson manifold.
\end{example}

\begin{definition}
A \emph{homogeneous Poisson-Nijenhuis manifold} is a Poisson-Nijenhuis manifold $(M, \pi, \phi)$ equipped with a \emph{homogeneity vector field} for $(\pi, \phi)$, i.e.~a vector field $\eta$ such that $\Ll_\eta \pi = - \pi$, and $\Ll_\eta \phi = 0$. The triple $(\pi, \phi, \eta)$ is called a \emph{homogeneous Poisson-Nijenhuis structure}. %, and we say that $\pi$ is homogeneous with respect to $\eta$.
\end{definition}

\begin{definition}\label{def:hgc}
A \emph{homogeneous generalized complex manifold} is a generalized complex manifold $(M, \Jj)$ equipped with a homogeneity vector field for
\[
\Jj = \left(
\begin{array}{cc}
\phi & \pi^\sharp \\
\omega_\flat & - \phi^\ast
\end{array}
\right)
\]
i.e.~a vector field $\eta$ such that, 1) $\Ll_\eta \pi = -\pi$, 2) $\Ll_\eta \phi = 0$, and 3) $\Ll_\eta \omega = \omega$. The pair $(\Jj, \eta)$ is called a \emph{homogeneous generalized complex structure}.
\end{definition}

\begin{theorem}\label{theor:hhP_hPN_hgc}
Let $X = (M,j)$ be a complex manifold, let $\Pi \in \Gamma (\wedge^2 (TM)^\C)$ be a complex bi-vector on $M$, and let $H \in \Gamma ((TM)^\C)$ be a complex vector field. Denote by $\pi',\pi$ the real and the imaginary part of $\Pi$ respectively: $\Pi = \pi' + i \pi$, $\pi',\pi \in \Gamma (\wedge^2 TM)$. Finally, let $\eta$ and $\eta'$ be twice the real and the imaginary part of $H$ respectively: $H = \frac{1}{2} \left(\eta + i \eta' \right)$, $\eta, \eta' \in \mathfrak X (M)$. Then, the following three conditions are equivalent
\begin{enumerate}
\item $(X,\Pi, H)$ is a  homogeneous holomorphic Poisson manifold,
\item $(M, \pi, j, \eta)$ is a homogeneous Poisson-Nijenhuis manifold, $\pi' = \pi_j$, and $\eta' = -j\eta$.
\item $(M, \Jj, \eta)$, where $\Jj := \left(
\begin{smallmatrix}
j & \pi^\sharp \\
0 & - j^\ast
\end{smallmatrix}
\right)
$, is a homogeneous generalized complex manifold, $\pi' = \pi_j$, and $\eta' = -j\eta$.
\end{enumerate}  
\end{theorem}

\begin{proof} \

$(1) \Leftrightarrow (2)$. From Theorem \ref{theor:hP_PN_gc}, $(X,\Pi)$ is a holomorphic Poisson manifold if and only if $(M, \pi, j)$ is a Poisson-Nijenhuis manifold and $\pi' = \pi_j$. Moreover, $H$ is a holomorphic vector field if and only if it is a section of $T^{1,0} M$, whence $\eta' = -j \eta$, $\overline{\partial} H = 0$, and $\Ll_\eta j = 0$. Now, let $\Pi = \pi_j + i \pi$ be a holomorphic Poisson bi-vector and $H = \frac{1}{2}\left(\eta -ij \eta\right)$ a holomorphic vector field on $X$. It remains to check that $\Ll_H \Pi = - \Pi$ if and only if $\Ll_\eta \pi = -\pi$. To see this, for any complex multivector field $Z \in \Gamma (\wedge^\bullet (T M)^\C)$, let $Z^{k,l}$ be its projection onto $\Gamma (\wedge^k T^{1,0} M \oplus \wedge^l T^{0,1} M)$. Let $\overline H = \frac{1}{2} \left (\eta +ij \eta \right) \in \Gamma (T^{0,1} M)$ be the complex conjugate of $H$ and notice that,
\[
\Ll_{\overline H} \Pi = \overline{\partial}_{\overline{H}} \Pi + (\mathcal L_{\overline{H}} \Pi)^{1,1} + (\mathcal L_{\overline{H}} \Pi)^{0,2}.
\]
The latter expression vanishes identically. Indeed the first summand vanishes because $\Pi$ is holomorphic, the second summand vanishes because $H$ is holomorphic, hence $\overline{H}$ is anti-holomorphic. The third summand vanishes because $\Pi \in \Gamma (\wedge^2 T^{1,0} M)$ (use, e.g., local coordinates). It follows that
\begin{equation}\label{eq:im_L_Hbar}
\Ll_\eta \pi + \Ll_{j \eta} \pi_j = 2 \operatorname{Im}(\mathcal L_{\overline H} \Pi) = 0.
\end{equation}
Now, since $\Pi$ and $H$ are holomorphic, then $\Ll_H \Pi$ is holomorphic as well. In particular $\Ll_H \Pi$ and $-\Pi$ agree if and only if their imaginary parts agree. Finally
\[
\operatorname{Im}(\mathcal L_H \Pi) = \frac{1}{2} \left( \mathcal L_\eta \pi - \mathcal L_{j\eta} \pi_j\right) = \Ll_\eta \pi,
\]
where we used (\ref{eq:im_L_Hbar}). Hence $\mathcal L_H \Pi = - \Pi$ if and only if $\Ll_\eta \pi = - \pi$. This concludes the proof.

$(2) \Leftrightarrow (3)$. It immediately follows from \cite[Proposition 2.2]{C2011}.
\end{proof}

\begin{remark}
Let $(X, \Pi, H)$ be a homogeneous holomorphic Poisson manifold, with $X = (M,j)$, $\Pi = \pi_j + i \pi$ and $H = \frac{1}{2} (\eta -ij \eta)$. Then $(\pi, \pi_j, \eta)$ is a \emph{homogeneous bi-Hamiltonian structure}, i.e.~$(\pi, \pi_j)$ is a bi-Hamiltonian structure and $\eta$ is a homogeneity vector field for both $\pi$ and $\pi_j$. Additionally, looking at the real part of the identities $\Ll_H \Pi = - \Pi$ and $\Ll_{\overline H} \Pi = 0$ one easily sees that
\begin{equation}\label{eq:L_jeta_pi}
\Ll_{j\eta} \pi_j = \pi, \quad \text{and} \quad \Ll_{j\eta} \pi = - \pi_j.
\end{equation}
\end{remark}
  
\subsection{Holomorphic vector bundles, linear complex structures and the real gauge algebroid}

In the next subsection, we discuss the \emph{holomorphic gauge algebroid} of a holomorphic vector bundle (Definition \ref{def:holom_gauge}) from the real geometric point of view. It turns out that this description is simpler if one first revisits holomorphic vector bundles in terms of \emph{linear complex structures} on vector bundles (over a complex manifold). This is done in Lemma \ref{prop:holom}. In its turn, the holomorphic gauge algebroid plays a key role for holomorphic Jacobi and related structures (much as the real gauge algebroid plays a key role for standard Jacobi and related structures \cite{LOTV2014, V2015, VW2016}).

Let $X = (M,j)$ be a complex manifold. Both $T^{1,0}M$ and $T^{0,1}M$ are complex Lie algebroids, and a \emph{holomorphic vector bundle} $E \to X$ over $X$ can be seen as a complex vector bundle $E \to M$ equipped with a flat $T^{0,1}M$-connection. In particular there is an operator $\overline \partial : \Gamma (E) \to \Omega^{0,1} (X, E)$ whose kernel consists of \emph{holomorphic sections of $E$} (see, e.g., \cite{LSX2008}). 

\begin{definition}\label{Def:HolomLieAlg}
A \emph{holomorphic Lie algebroid} over $X$ is a holomorphic vector bundle $A \to X$ equipped with a \emph{holomorphic anchor} i.e.~a holomorphic vector bundle morphism $\rho: A \to TX$, and a $\C$-linear Lie bracket $[-,-] : \Gamma_A \times \Gamma_A \to \Gamma_A$ on its sheaf $\Gamma_A$ of holomorphic sections such that
\[
[\alpha_1, f\alpha_2] = \rho(\alpha_1) (f) \alpha_2 + f [\alpha_1, \alpha_2]  
\]
for all $\alpha_1, \alpha_2 \in \Gamma_A$ and $f \in \Oo_X$.
\end{definition}

In \cite{LSX2008} the authors prove that a holomorphic Lie algebroid $A \to X$ is equivalent to a holomorphic
vector bundle $A \to X$ equipped with a holomorphic vector bundle map $\rho : A \to TX$ and a \emph{real} Lie
algebroid structure, with anchor $\rho$ itself, and Lie bracket such that it restricts to a $\C$-bilinear Lie bracket on holomorphic sections. In what follows, given a holomorphic Lie algebroid $A \to X$, we denote by $A_{\mathrm{Re}} \to M$ the underlying real Lie algebroid. Now, denote by $j_A : A_{\mathrm{Re}} \to A_{\mathrm{Re}}$ the fiber-wise complex structure on $A_{\mathrm{Re}}$. The (Lie algebroid) Nijenhuis torsion of $j_A$ vanishes, i.e.
\[
[j_A \alpha, j_A \beta] - [\alpha, \beta] -j_A[j_A \alpha, \beta] -j_A [\alpha, j_A \beta] = 0,
\]
for all $\alpha, \beta \in \Gamma (A_{\mathrm{Re}})$. From a torsionless endomorphism $\phi : A_{\mathrm{Re}} \to A_{\mathrm{Re}}$ of a real Lie algebroid $A_{\mathrm{Re}}$ one can always define a new Lie algebroid structure $(A_{\mathrm{Re}})_\phi$ on $A_{\mathrm{Re}}$ with anchor $\rho_\phi := \rho \circ \phi$ and bracket $[-,-]_\phi$ defined by
\[
[\alpha, \beta]_\phi := [\phi \alpha, \beta] + [\alpha, \phi \beta] - \phi [\alpha, \beta].
\]
In particular, a holomorphic Lie algebroid $A \to X$ defines another real Lie algebroid $(A_{\mathrm{Re}})_{j_A}$. We call $A_{\mathrm{Re}}$ and $(A_{\mathrm{Re}})_{j_A}$ the \emph{real} and the \emph{imaginary Lie algebroids} of $A \to X$ respectively.

Finally, complexifying $A_{\mathrm{Re}}$ and decomposing into the eigenbundles of the complex structure, one also defines from $A \to X$ two complex Lie algebroids $A^{1,0} \to M$ and $A^{0,1} \to M$ with anchors $\rho^{1,0} : A^{1,0} \to T^{1,0}M$ and $\rho^{0,1}: A^{0,1} \to T^{0,1}M$, and Lie brackets $[-,-]^{1,0}$ and $[-,-]^{0,1}$. We refer to \cite{LSX2008} for the details.

\begin{example}
The tangent bundle $TX \to X$ is a holomorphic Lie algebroid, with underlying real Lie algebroid $TM \to M$, imaginary Lie algebroid $(TM)_j \to M$, and associated complex Lie algebroids $T^{1,0}M$ and $T^{0,1}M$.
\end{example} 

Given a holomorphic vector bundle $E \to X$ one can define its \emph{holomorphic gauge algebroid} $DE \to E$ (Definition \ref{def:holom_gauge}), encoding infinitesimal automorphisms of $E$. As already mentioned, we do this in the next subsection. Now recall what the real gauge algebroid of a real vector bundle is. Given a real vector bundle $E \to M$, and a point $x \in M$, the fiber over $x$ of the real gauge algebroid of $E$ consists of $\R$-linear maps $\Delta : \Gamma (E) \to E_x$ satisfying the following Leibniz rule: $\Delta (fe) = \xi(f) e_x + f(x) \Delta (e)$, for all $e \in \Gamma (E)$, all $f \in C^\infty (M)$, and a (necessarily unique) tangent vector $\xi \in T_x M$. We denote by $D_\R E$ the real gauge algebroid of $E \to M$ to distinguish it from the holomorphic gauge algebroid of $E \to X$ (that will be simply denoted by $DE$). Sections of $D_\R E$ are \emph{derivations of $E$} (also called \emph{covariant differential operators} in \cite{M2005}, see also \cite{Kosmann-Mackenzie}), i.e.~$\R$-linear operators $\Delta : \Gamma (E) \to \Gamma (E)$ satisfying the following \emph{Leibniz rule}: 
\begin{equation}\label{eq:gen_Leib}
\Delta (fe) = \xi(f) e + f \Delta (e), 
\end{equation}
for all $e \in \Gamma (E)$, all $f \in C^\infty (M)$, and a (necessarily unique) vector field $\xi \in \mathfrak X (M)$, called the \emph{symbol} of $\Delta$, and usually denoted $\sigma (\Delta)$.

\begin{remark}
Notice that the terminology ``\emph{derivation}'' for a section of $D_\R E$ (and like-wise ``complex derivation'' below) does not mean that sections of $E$ form an algebra. It only reflects the (generalized) Leibniz rule (\ref{eq:gen_Leib}) and is now rather standard in the literature.
\end{remark}

The gauge algebroid is a Lie algebroid with anchor given by the symbol $\sigma : DE \to TM$, and Lie bracket given by the commutator of derivations. The kernel of the symbol consists of $C^\infty (M)$-linear derivations, i.e.~endomorphisms of $E$ covering the identity of $M$. Hence there is a short exact sequence of vector bundles:

\[
0 \longrightarrow \operatorname{End}_\R E \longrightarrow D_\R E \overset{\sigma}{\longrightarrow} TM \longrightarrow 0.
\]

For more details about the gauge algebroid (including its functorial properties) we refer to \cite{LOTV2014,V2015}. We only recall here that sections of the gauge algebroid $DE \to M$ are in one-to-one correspondence with infinitesimal automorphisms of $E$, or, equivalently, \emph{linear vector fields} on (the total space of) $E$, i.e.~vector fields $\xi \in \mathfrak X (E)$ preserving fiber-wise linear functions on $E$. We denote by $\mathfrak X_{\mathrm{lin}} (E)$ the Lie algebra of linear vector fields on $E$. Every linear vector field $\xi \in \mathfrak X_{\mathrm{lin}} (E)$ projects to a vector field $\underline \xi \in \mathfrak X (M)$, and the linear vector field $\xi \in \mathfrak X (E)$ corresponds to the derivation $\Delta_\xi : \Gamma (E) \to \Gamma(E)$ implicitly defined by
\[
\langle \varphi , \Delta_\xi e \rangle := \underline \xi \langle \varphi, e \rangle - \langle \xi (\varphi), e \rangle, 
\]
for all $\varphi \in \Gamma(E_\R^\ast)$ and $e \in \Gamma(E)$ (here we identified sections of the real dual vector bundle $E_\R^\ast$ of $E$ with fiber-wise linear functions on $E$). The correspondence $\xi \mapsto \Delta_\xi$ is a Lie algebra isomorphism and, additionally, $\sigma (\Delta_\xi) = \underline \xi$. Finally, $C^\infty (M)$-linear derivations of $E$, i.e~endomorphisms, correspond to vertical linear vector fields on $E$.

Before giving a precise definition of the holomorphic gauge algebroid it is convenient to discuss \emph{linear $(1,1)$ tensors on a vector bundle}. Thus, let $E \to M$ be a vector bundle. Recall that $TE$ is a double vector bundle, with side vector bundles $E$ and $TM$ (see, e.g., \cite[Chapter 9]{M2005}).

\begin{definition}
A $(1,1)$ tensor $\phi : TE \to TE$ is \emph{linear} if it is a double vector bundle morphism, or, equivalently, if it preserves linear vector fields on $E$. 
\end{definition}

\begin{lemma}\label{prop:linear_end}
There is a $C^\infty (M)$-linear one-to-one correspondence $\phi \mapsto \phi_{DE}$, between linear $(1,1)$ tensors $\phi$ on $E$ and endomorphisms $\psi : DE \to DE$ with the following two properties:
\begin{enumerate}
\item there is a $(1,1)$ tensor $\underline \psi : TM \to TM$ such that $\underline \psi \circ \sigma = \sigma \circ \psi$, and
\item there is an endomorphism $\psi_E \in \Gamma (\operatorname{End}_\R E)$ such that $\psi (h) = \psi_E \circ h$ for every endomorphism $h \in \Gamma (\operatorname{End}_\R E)$. 
\end{enumerate} 
The correspondence $\phi \mapsto \phi_{DE}$ preserves the compositions, i.e.~$(\phi \circ \phi')_{DE} = \phi_{DE} \circ \phi'_{DE}$ for every two linear $(1,1)$ tensors on $E$.
\end{lemma}

\begin{proof}
It immediately follows from the definition that a linear $(1,1)$ tensor $\phi : TE \to TE$ defines an endomorphism $\phi_{DE} : D_\R E \to D_\R E$  by restriction to linear vector fields. Since $\phi$ is a morphism of double vector bundles, it descends to a $(1,1)$ tensor $\underline \phi : TM \to TM$ on $M$, and preserves vertical tangent vectors. Hence $\phi_{DE}$ has property (1) in the statement. Additionally, it satisfies $\phi_{DE} (h) \in  \Gamma (\operatorname{End}_\R E)$ for all $h \in \Gamma (\operatorname{End}_\R E) \subset \Gamma (D_\R E)$. Now, we define: $\phi_E := \phi_{DE} (\operatorname{id}_E) \in \Gamma (\operatorname{End}_\R E)$. It is then easy to see that, actually, $\phi_{DE} (h) = \phi_E \circ h$ for all $h \in  \Gamma (\operatorname{End}_\R E)$.
Hence $\phi_DE$ has also property (2). Since linear vector fields generate the whole $\mathfrak X (E)$, at least away from the zero section of $E \to M$, then the correspondence $\phi \mapsto \phi_{DE}$ is injective. Surjectivity is easily checked in local coordinates. The last part of the proposition is obvious.
\end{proof}

\begin{remark}\label{rem:linear_Nijenhuis}
Since the isomorphism $\mathfrak X_{\operatorname{lin}} (E) \to \Gamma (D_\R E)$ intertwines the Lie brackets, and linear vector fields generate all vector fields on $E$, away from the zero section of $E \to M$, then a linear $(1,1)$ tensor $\phi : TE \to TE$ is torsionless if and only if $\phi_{DE} : D_\R E \to D_\R E$ is torsionless.
\end{remark}

Now we need to further describe features of holomorphic vector bundles. This is done in Lemma \ref{prop:holom} below:

\begin{lemma}\label{prop:holom}
Let $M$ be a smooth manifold. The following data are equivalent:
\begin{enumerate}
\item a complex structure $j$ on $M$ and a holomorphic vector bundle $E \to X = (M,j)$;
\item a complex structure $j$ on $M$ and a complex vector bundle $E \to X = (M,j)$ equipped with a flat $T^{0,1}M$ connection;
\item a real vector bundle $E \to M$ equipped with an integrable linear complex structure $j^{\mathrm{tot}}_E : TE \to TE$ in its total space;
\item a complex structure $j$ on $M$ and a complex vector bundle $E \to M$, with fiber-wise complex structure $j_E : E \to E$, equipped with a torsionless complex structure $j_{DE}$ on its real gauge algebroid $D_\R E$ such that
\begin{enumerate}
	\item[(4b)] the symbol $\sigma : D_\R E \to TM$ intertwines $j_{DE}$ and $j$,
	\item[(4c)] the restriction of $j_{DE}$ to endomorphisms $\operatorname{End}_\R E$ agrees with the map $\operatorname{End}_\R E \to \operatorname{End}_\R E$, $\phi \mapsto j_E \circ \phi$.
	\end{enumerate}
\end{enumerate}
\end{lemma}

\begin{proof}

The equivalence $(1) \Leftrightarrow (2)$ is standard and has already been mentioned at the beginning of this section.
The equivalence $(1) \Leftrightarrow (3)$ is clear: if $E \to X$ is a holomorphic vector bundle, in particular $E$ is a complex manifold. If we denote by $j^{\mathrm{tot}}_E$ the complex structure on $E$, the linearity can be checked in local coordinates. For the converse, one can argue in a similar way and we leave the details to the reader.  
Finally $(3) \Leftrightarrow (4)$ immediately follows from Lemma \ref{prop:linear_end} and Remark \ref{rem:linear_Nijenhuis}.

 \end{proof}

 \begin{remark}

 The equivalence $(2) \Leftrightarrow (3)$ in Lemma \ref{prop:holom} can also be proved directly. For some purposes, this proof is useful, and we provide it here. As a byproduct, we will have a description of the $T^{0,1}M$-connection in $E$ in terms of the complex structure $j^{\mathrm{tot}}_E$, and vice-versa. 
 
 So, begin with a complex vector bundle $E \to M$ over a complex manifold $X = (M,j)$, and a flat $T^{0,1}M$-connection $\overline \partial$ in $E$. We denote by $j_E : E \to E$ the fiber-wise complex structure on $E$. Let $E^\ast \to M$ be the complex dual to $E$. The connection $\overline \partial$ induces a flat connection in $E^\ast$, also denoted by $\overline \partial$. In order to define $j_E^{\mathrm{tot}}$, denote by $\pi : E \to M$ the projection, let $e \in E$ and $x = \pi (e)$. Every tangent vector $\xi \in T_e E$ is a derivation $\xi : C^\infty (E) \to \mathbb R$, and can be extended to a self-conjugate  derivation, also denoted by $\xi : C^\infty (E, \C) \to \C$, by $\C$-linearity. On the other hand, every self-conjugate complex derivation $\xi : C^\infty (E, \C) \to \C$ determines a tangent vector $\xi \in T_e E$ by restriction to real functions. A self-conjugate derivation $\xi : C^\infty (E, \C) \to \C$ is completely determined by its action on fiber-wise constant functions, i.e.~functions in $C^\infty (M, \C)$, and on fiber-wise linear functions, i.e.~sections of $E^\ast$ (notice, however, that this is not so for non self-conjugate derivations). Conversely, a triple $(e, \underline \eta, \eta)$ consisting of 
\begin{enumerate}
\item a point $e \in E$,
\item a tangent vector $\underline \eta \in T_x M$, $x = \pi (e)$, and
\item a $\C$-linear operator $\eta : \Gamma (E^\ast) \to \C$,
\end{enumerate}
such that
\begin{equation}\label{eq:LR}
\eta (f \varphi) = \underline \eta (f) \langle \varphi_x , e \rangle + f(x) \eta (\varphi),
\end{equation}
for all $f \in C^\infty (M, \C)$ and $\varphi \in \Gamma (E^\ast)$, comes from a unique $\xi \in T_e E$ such that $\underline \eta = (d \pi) (\xi)$, and $\eta$ is the restriction of $\xi : C^\infty (E, \mathbb C) \to \mathbb C$ to fiber-wise linear complex functions. So let $e \in E$, and $\xi \in T_e E$, and define 
\[
\underline \eta   \in T_x M, \quad \text{and} \quad
\eta  : \Gamma (E^\ast) \to \C, 
\]
as follows. Firstly, write $\underline \eta := j(\pi_\ast \xi)$. Write also $(\pi_\ast \xi)^{0,1}:= p^{0,1} (\pi_\ast \xi) = (\pi_\ast \xi + i\underline \eta)/2$ where $p^{0,1} : T_x M \otimes_\R \C \to T^{0,1}_x X$ is the projection. Secondly,  write
\begin{equation}\label{eq:eta}
%\eta (\varphi) := \xi (j^\ast_E \varphi ) - 2 \langle \overline \partial_{\underline \eta^{0,1}} \varphi, j_E e\rangle
\eta (\varphi) := i \left( \xi (\varphi ) - 2 \langle \overline \partial_{(\pi_\ast \xi)^{0,1}} \varphi,  e\rangle \right)
\end{equation}
for all $\varphi \in \Gamma (E^\ast)$. An easy computation shows that $\eta$ satisfies the Leibniz rule (\ref{eq:LR}). Hence $(e, \underline \eta, \eta)$ defines a tangent vector $\xi' \in T_e E$. We write $j_E^{\mathrm{tot}} \xi := \xi'$. Checking that $(j_E^\mathrm{tot})^2 = - \mathbb 1$ is straightforward. Thus we have an almost complex structure $j_E^\mathrm{tot}: TE \to TE$ on $E$. Now, the integrability of $j_E^\mathrm{tot}$ follows from the flatness of $\underline \partial$. The linearity of $j_E^\mathrm{tot}$ immediately follows from (\ref{eq:eta}) and the fact that $d\pi$ maps $\xi'$ to $\underline \eta$. In particular, $j_E^{\mathrm{tot}}$ descends to $j$ under $d\pi : TE \to TM$ and agrees with $j_E$ on fibers.

Conversely, let $j_E^\mathrm{tot} : TE \to TE$ be a linear complex structure. Then it descends to a complex structure $j$ on $M$ under $d\pi : TE \to TM$ and restricts to a fiber-wise complex structure $j_E : E \to E$ on fibers.  Define a $T^{0,1}M$-connection in $E^\ast$ as follows.  For $\varphi \in \Gamma (E^\ast)$, and $\underline \eta^{0,1} \in T^{0,1} M$, we set:
\begin{equation}\label{Eq:overlinebar}
\langle \overline \partial_{\underline \eta^{0,1}} \varphi, e\rangle = \tilde \xi(\varphi ),
\end{equation}
where $\tilde \xi \in T^{0,1} E$ is any tangent vector such that $\pi_\ast \tilde \xi = \underline \eta^{0,1}$. The $\C$-linearity of $\varphi$ guarantees that $\overline \partial_{\underline \eta^{0,1}} \varphi$ is independent of the choice of $\tilde \xi$. Now, the linearity in the argument $e$ follows from the linearity of $j_E^\mathrm{tot}$.  The linearity in the argument $\underline \eta^{0,1}$ is obvious, and the Leibniz rule with respect to the argument $\varphi$ follows from the fact that $d\pi$ intertwines $j_E^\mathrm{tot}$ and $j$ (by definition of $j$). The flatness of $\overline \partial$ is equivalent to  the integrability of $j_E^{\mathrm{tot}}$. So $\overline \partial$ is a flat $T^{0,1} M$-connection in $E^\ast$. By duality, it induces a flat $T^{0,1} M$-connection in $E$. Comparing (\ref{eq:eta}) and (\ref{Eq:overlinebar}) we see that this construction inverts the above contruction of $j_e^{\mathrm{tot}}$ from $\overline \partial$. 
\end{remark}
 
 \begin{remark}\label{rem:j_DE}
 A direct computation exploiting Equation (\ref{eq:eta}) shows that 
 \begin{equation}\label{eq:j_DE}
 j_{DE} (\Delta) = j_E \circ (\Delta - 2 \overline \partial_{\sigma (\Delta)^{0,1}}),
 \end{equation}
 for all $\Delta \in D_\R E$. Formula (\ref{eq:j_DE}) can be used to prove directly the equivalence between (2) and (4) in Lemma \ref{prop:holom}. Notice that, despite $ j_{DE} \Delta \in D_\R E$ for all $\Delta \in D_\R E$, none of the two summands in the right hand side of (\ref{eq:j_DE}) is in $D_\R E$. Finally, it immediately follows from Formula (\ref{eq:j_DE}) that $j_{DE}$ preserves $\C$-linear sections of $D_\R E$, i.e.~those sections commuting with $j_E$.
\end{remark}

\subsection{The holomorphic gauge algebroid}
 
Let $E \to X$ be a holomorphic vector bundle over a complex manifold $X = (M,j)$.  Lemma (\ref{prop:holom}) shows that the gauge algebroid $D_\R E$ is equipped with a torsionless complex structure $j_{DE}$. Additionally $j_{DE}$ restricts to the subbundle $DE$ consisting of $\C$-linear derivations (see Remark \ref{rem:j_DE}). We will show below that $DE$ is a holomorphic Lie algebroid over $X$, whose underlying real Lie algebroid structure $(DE)_{\mathrm{Re}}$ is obtained from $D_\R E$ by restriction. To see this, we first describe $DE$ in an alternative way.

First of all, sections of the complexified tangent bundle $(T M)^\C$ can be seen as derivations of the complex algebra $C^\infty (M, \C)$. They are also real derivations $\xi$ of the real vector bundle $\C_M := M \times \C  \to M$ such that $X(1) = X(i) = 0$, and we denote them by $\mathfrak X_\C (M)$. Clearly, $(T M)^\C$ is a complex Lie algebroid whose anchor is the identity and whose Lie bracket is the commutator. Now, let $E \to M$ be a complex vector bundle over a smooth manifold (beware not yet a holomorphic vector bundle over a complex manifold). Denote by $D_{\C} E \to M$ the bundle whose sections are $\C$-linear operators $\Delta : \Gamma (E) \to \Gamma (E)$ such that there exists (a necessarily unique) $\sigma (\Delta) \in \mathfrak X_\C (M)$ such that:
\[
\Delta (f e)= \sigma (\Delta)(f) e + f \Delta (e),
\]
for all $f \in C^\infty (M, \C)$, and all $e \in \Gamma (E)$. Sections of the vector bundle $D_\C E \to X$ are called \emph{complex derivations} of $E$. Clearly, $D_\C E$ is a complex Lie algebroid with anchor $\sigma$ and Lie bracket the commutator of complex derivations. The complex structure on the vector bundle $D_\C E \to M$ is the composition with $j_E : E \to E$. The kernel of the symbol map $\sigma : D_\C E \to (TM)^\C$ is the bundle of $\C$-linear endomorphisms of $E$. In sum,  there is a short exact sequence of  complex vector bundles:
\[
0 \longrightarrow \operatorname{End}_\mathbb C E \longrightarrow D_\C E \overset{\sigma}{\longrightarrow} (T M)^\C \longrightarrow 0,
\] 
where $\operatorname{End}_\mathbb C E$ denotes (fiber-wise) $\C$-linear endomorphisms of $E$.

\begin{remark}
Notice that complex derivations of $E$ are not \emph{real} derivations of $E$, i.e. derivations of $E$, when considered as a real vector bundle. More precisely, a complex derivation is a real one if and only if its symbol is a real vector field.
\end{remark}

Now, we assume that $M$ is equipped with a complex structure $j$ and write $X = (M,j)$. In this case, $(TM)^\C = T^{1,0} M \oplus T^{0,1} M$. Denote $D^{1,0} E := \sigma^{-1} (T^{1,0} M)$ and $  D^{0,1} E := \sigma^{-1} (T^{0,1} M)$. We have:
\begin{equation}\label{eq:1}
D_\C E =   D^{1,0} E  +   D^{0,1} E,
\end{equation}
but $  D^{1,0} E \cap   D^{0,1} E = \operatorname{End}_\C E$, so (\ref{eq:1}) is not a direct sum decomposition. However, we can ``correct it'' to a direct sum decomposition if $E$ is a holomorphic vector bundle over $X$, i.e.~if there is a flat $T^{0,1}M$-connection $\overline \partial$ in $E$. Indeed, the connection $\overline{\partial}$ splits the short exact sequence $0 \to \operatorname{End}_\C E \to   D^{0,1} E \to T^{0,1} M \to 0$, hence there is a direct sum decomposition:
\begin{equation}\label{eq:split_D}
D_\C E =   D^{1,0} E \oplus T^{0,1} M,
\end{equation}
and the formula
\[
\overline{\partial}{}^{DE}_\xi \Delta := [\overline \partial_\xi, \Delta]^{1,0},
\]
where $\xi \in \Gamma (T^{0,1}M)$, and $\Delta \in \Gamma (D^{1,0}E)$, defines a flat $T^{0,1}M$ connection in $D^{1,0}E$, which makes it a holomorphic vector bundle over $X$. Here and i what follows  we denote by $\Delta \mapsto \Delta^{1,0}$ the projection $D_\C E \to   D^{1,0} E$ with kernel $\operatorname{im} \overline \partial \cong T^{0,1} M$. 

%there is a unique holomorphic Lie algebroid structure on $D^{\mathrm{hol}} E \to M$ with bracket $[-,-]^{\mathrm{hol}}$ agreeing with $[-,-]$ on holomorphic sections and anchor $\sigma^{\mathrm{hol}} : D^{\mathrm{hol}}E \to TM$, called the \emph{holomorphic symbol}, given by the composition of $\sigma :   D^{1,0} E \to T^{1,0} M$ followed by the complex vector bundle isomorphism $2 \operatorname{Re} : T^{1,0} M \to TM$.

\begin{lemma}\label{lem:D^1,0=D}
There is a natural complex vector bundle isomorphism $\iota :  D^{1,0} E \to DE$ such that: 
\begin{enumerate}
\item the diagram
\[
\begin{array}{c}
\xymatrix{ D^{1,0}E \ar[r]^-{\iota} \ar[d]_-{\sigma} & D E \ar[d]^-{\sigma} \\
                  T^{1,0} M \ar[r]^-{2 \operatorname{Re}}  & TM }
                  \end{array},
\]
commutes
\item 
\[
\iota [\Delta_1, \Delta_2] = [ \iota \Delta_1, \iota \Delta_2]
\] 
for every two \emph{holomorphic} sections of $D^{1,0} E \to X$.
\end{enumerate}
\end{lemma}
 
 \begin{proof}
Define $\iota : D^{1,0} E \to D E$ by
\[
\iota \Delta  :=  \Delta + \overline \partial_{\overline{\sigma (\Delta)}},
\]
where by $\overline \xi$ we mean the complex conjugate of a complex vector field $\xi \in \mathfrak X_\C (M)$. Clearly, $\iota$ is a well-defined homomorphism of complex vector bundles. It is easy to see that it is injective. Indeed it follows from $\Delta + \overline \partial_{\overline{\sigma (\Delta)}}  = 0$, that
\[
0 = \sigma (\Delta + \overline \partial_{\overline{\sigma (\Delta)}}) = 2 \operatorname{Re} (\sigma (\Delta)).
\]
Since $\sigma (\Delta) \in \Gamma (T^{1,0}M)$, this is enough to have $\sigma (\Delta) = 0$. Hence $\Delta = \Delta + \overline \partial_{\overline{\sigma (\Delta)}} = \iota \Delta = 0$. Since $DE$ and $ D^{1,0} E$ have the same complex rank $\dim_\C M + (\operatorname{rank}_\C E)^2$, then $\iota$ is an isomorphism.

Now, (1) follows from 
\[
\sigma (\iota \Delta) = \sigma(\Delta + \overline \partial_{\overline{\sigma (\Delta)}}) =  2 \operatorname{Re} (\sigma (\Delta)).
\]
Condition (2) can be easily checked by a similar direct computation.
\end{proof}

\begin{remark}
The inverse $\iota^{-1} : DE \to D^{1,0}E$ of the isomorphism $\iota$ is given by:
\[
\iota^{-1} \Delta = \Delta - \overline{\partial}_{\sigma(\Delta)^{0,1}} = \frac{1}{2} (\Delta - j_E \circ j_{DE} \Delta),
\]
where $\Delta \in \Gamma (DE)$.
\end{remark}
\begin{corollary}
There is a holomorphic Lie algebroid structure $DE \to X$ such that:
\begin{enumerate}
\item the underlying real Lie algebroid structure is induced from that of $D_\R E$ by restriction (of the anchor and the bracket),
\item the associated complex Lie algebroid $(DE)^{1,0}$ is canonically isomorphic to $D^{1,0}E$.
\end{enumerate}
\end{corollary}

\begin{definition}\label{def:holom_gauge}
The holomorphic Lie algebroid $DE \to X$ is  called the \emph{holomorphic gauge algebroid} of the holomorphic vector bundle $E$.
\end{definition}

\begin{remark}
Holomorphic sections of $DE \to X$ are precisely derivations of the sheaf of holomorphic sections of $E \to X$.
\end{remark}

\begin{remark}
The holomorphic gauge algebroid $DE$, its associated complex algebroid $ D^{1,0} E$, and the \emph{complex gauge algebroid} $D_\C E$ fit in the following triangle of short exact sequences:
\[ 
\begin{array}{c}
\xymatrix@C=7pt{ & 0 \ar[rr] & & \operatorname{End}_{\mathbb C} E \ar[rr] \ar@{=}[dd] \ar@{=}[dl]  & & D^{1,0} E \ar[rr] \ar[dd]^(.3){\iota}  & & T^{1,0}M \ar[rr] \ar[dd]^(.3){2 \operatorname{Re}} & & 0 \\
0 \ar[rr] & & \operatorname{End}_{\mathbb C} E \ar[rr]|!{[ur];[dr]}\hole \ar@{=}[dr] & & D_\C E \ar[rr]|!{[ur];[dr]}\hole \ar[ur] & & (TM)^\C \ar[rr]|!{[ur];[dr]}\hole \ar[ur] & & 0 &  \\
& 0 \ar[rr] & & \operatorname{End}_{\mathbb C} E \ar[rr] & & D E \ar[rr] \ar[ul]  & & TX \ar[rr]  \ar[ul] & & 0}
\end{array}, 
\]
where the vertical arrows are isomorphisms. Clearly, the bottom row is also isomorphic to $0 \to \operatorname{End}_\mathbb C E \to D^{0,1} E \to T^{0,1} M \to 0$. Finally, the decomposition (\ref{eq:split_D}) gives $(D^{1,0}E, T^{0,1}M)$ the structure of a \emph{matched pair of Lie algebroids}. The latter is an instance of the matched pair of Lie algebroids associated to a holomorphic Lie algebroid $A$ \cite{LSX2008}. In the present case $A = DE$.
\end{remark}

\subsection{Holomorphic jets}
In this section we briefly discuss first holomorphic jets of a holomorphic vector bundle from the real differential geometric point of view. In the theory of Jacobi structures the first jet bundle plays a \emph{dual role} to the gauge algebroid. First of all, let $E \to M$ be a real vector bundle. We denote by $\mathfrak J^1_\R E$ the first (real) jet bundle of $E \to M$, and by $\mathfrak j^1_\R : \Gamma (E) \to \Gamma(\mathfrak J^1_\R E)$, $e \mapsto \mathfrak j^1_\R e$ the first (real) jet prolongation of sections of $E \to M$. Recall that $\mathfrak J^1_\R E$ fits in the short exact sequence of vector bundles over $M$:
\[
0 \longrightarrow TM \otimes_\R E \overset{\gamma}{\longrightarrow} \mathfrak J^1_\R E \longrightarrow E \longrightarrow 0,
\]
where the second arrow is the embedding $\gamma : TM \otimes_\R E \to \mathfrak J^1_\R E$, $d f \otimes e \mapsto \mathfrak j^1_\R (fe) - f \mathfrak j^1_{\mathbb R} e$, for all $f \in C^\infty (M)$, and $e \in \Gamma(E)$. Additionally, every section $\theta \in \Gamma(\mathfrak J^1_\R E)$ can be uniquely written in the form
\[
\theta = \mathfrak j^1_\R e + \gamma (\omega),
\]
where  $e \in \Gamma(E)$, and $\omega \in \Gamma (T^\ast M \otimes_\R E)$.

Now let $X = (M,j)$ be a complex manifold and let $E \to X$ be a holomorphic vector bundle over it. In particular, $E \to M$ is a complex vector bundle, and $\mathfrak J^1_\R E \to M$ inherits a (fiber-wise) complex structure from it: $i \cdot \mathfrak j^1_\R e := \mathfrak j^1_\R ie$, for all $e \in \Gamma (E)$.

\begin{remark}
The vector bundle $\mathfrak J^1_\R E \to M$ is canonically equipped with another complex vector bundle structure $j_{\mathfrak J^1 E} : \mathfrak J^1_\R E \to \mathfrak J^1_\R E$. To see this, first notice that there is a well-defined embedding $(T^{0,1}M)^\ast \otimes E \to T^\ast M \otimes _\R E$ given by the composition:
\[
(T^{0,1} M)^\ast \otimes E  \longrightarrow (T^\ast M \otimes_\R \C) \otimes E \overset{\cong}{\longrightarrow} T^\ast M \otimes_\R E.
\]
The embedding $(T^{0,1}M)^\ast \otimes E \to T^\ast M \otimes _\R E$ is right inverse to the projection $T^\ast M \otimes_\R E \to (T^{0,1}M)^\ast \otimes E$, $\omega \otimes_\R e \mapsto \frac{1}{2} (\omega + ij^\ast \omega)\otimes e$. In the following we will understand this embedding and interpret $(T^{0,1}M)^\ast \otimes E$ as a subbundle of $T^\ast M \otimes _\R E$. Then $j_{\mathfrak J^1 E}$ is defined on sections of $\mathfrak J^1_\R E$ by
\[
j_{\mathfrak J^1 E} (\mathfrak j^1_\R e + \gamma (\omega)) :=  \mathfrak j^1_\R (ie) + \gamma (j^\ast \omega - 2i \overline \partial e).
\]
\end{remark}

The connection $\overline \partial : \Gamma (E) \to \Gamma ((T^{0,1}M)^\ast \otimes E)$ determines a morphism $\Phi_{\overline \partial} : \mathfrak J^1_\R E \to (T^{0,1}M)^\ast \otimes E$, $\mathfrak j^1_\R e \mapsto \overline \partial e$, of real vector bundles.

\begin{definition}
The vector bundle $\mathfrak J^1 E := \ker \Phi_{\overline \partial}$ is the bundle of \emph{first holomorphic jets} of sections of $E \to X$.
\end{definition}

The above definition reflects the idea that \emph{holomorphic jets} are \emph{jets of holomorphic sections}. In this spirit, obviously, holomorphic jets can be introduced without any reference to real jets. If one does so, it immediately follows that the bundle of holomorphic jets is a holomorphic vector bundle (see also Proposition \ref{prop:J^1E_hol} below). What is not obvious is the description of the associated $T^{0,1} M$-connection, that we now provide.

First of all $\mathfrak J^1 E$ is a complex subbundle of $\mathfrak J^1_\R E$. Secondly, the inclusion $(T^{0,1}M)^\ast \otimes E \to \mathfrak J^1_\R E$ given by the composition
\[
(T^{0,1}M)^\ast \otimes E \longrightarrow T^\ast M \otimes_\R E \overset{\gamma}{\longrightarrow} \mathfrak J^1_\R E
\]
splits the exact sequence
\[
0 \longrightarrow \mathfrak J^1 E \longrightarrow \mathfrak J^1_\R E \overset{\Phi_{\overline \partial}}{\longrightarrow} T^{0,1}M \otimes E \longrightarrow 0.
\]
Hence there is a direct sum decomposition 
\begin{equation}\label{eq:split_jet}
\mathfrak J^1_\R E = \mathfrak J^1 E \oplus (T^{0,1}M)^\ast \otimes E.
\end{equation}
In particular, there are a projection $\mathfrak J^1_\R E \to \mathfrak J^1 E$, denoted $\theta \mapsto \theta^{1,0}$, and a short exact sequence
\[
0 \longrightarrow (T^{1,0} M)^\ast \otimes E \overset{\gamma^{1,0}}{\longrightarrow} \mathfrak J^1 E \longrightarrow E \longrightarrow 0,
\]
where $\gamma^{1,0}$ is given by the composition
\[
(T^{1,0} M)^\ast \otimes E \longrightarrow (T^\ast M \otimes_\R \C) \otimes E \overset{\cong}{\longrightarrow} T^\ast M \otimes_\R E \overset{\gamma}{\longrightarrow} \mathfrak J_\R^1 E \longrightarrow \mathfrak J^1 E.
\]
Denote by $\mathfrak j^{1,0} : \Gamma (E) \to \Gamma(\mathfrak J^1 E)$ the composition of the first jet prolongation $\mathfrak j^1_\R : \Gamma (E) \to \Gamma(\mathfrak J^1_\R E)$ followed by the projection $\Gamma (\mathfrak J^1_\R E) \to \Gamma (\mathfrak J^1 E)$, i.e.~$\mathfrak j^{1,0}e = (\mathfrak j^1_\R e)^{1,0} = \mathfrak j^1_\R e - \gamma (\overline \partial e)$, for all $e \in \Gamma (E)$. Let $e \in \Gamma (E)$, and $f \in C^\infty (M, \C)$. Then
\begin{equation}\label{eq:gamma^1,0}
\gamma^{1,0}(df \otimes e) = \mathfrak j^{1,0} (fe) - f \mathfrak j^{1,0} e.
\end{equation}
Moreover, it is clear that every section $\theta \in \Gamma(\mathfrak J^1 E)$ can be uniquely written in the form
\[
\theta = \mathfrak j^{1,0} e + \gamma^{1,0} (\omega),
\]
where $e \in \Gamma(E)$, and $\omega \in \Gamma ((T^{1,0} M)^\ast \otimes E)$.

We are now in the position to define a flat $T^{0,1}M$-connection $\overline \partial{}^{\mathfrak J^1 E}$ on $\mathfrak J^1 E$. Namely,  write
\begin{equation}\label{eq:flat_conn_J^1}
\overline \partial{}^{\mathfrak J^1 E}_\xi ( \mathfrak j^{1,0} e + \gamma^{1,0} (\omega)) :=  \mathfrak j^{1,0} \left( \overline\partial_\xi e \right) + \gamma^{1,0} \left( \overline \partial_\xi \omega - i_{\partial \xi} \overline \partial e \right),
\end{equation}
for all $e \in \Gamma (E)$, $\omega \in \Gamma ((T^{0,1}M)^\ast \otimes E)$, and $\xi \in \Gamma (T^{1,0}M)$. A direct computation exploiting (\ref{eq:gamma^1,0}) shows that $\overline \partial{}^{\mathfrak J^1 E}$ is a well-defined flat $T^{0,1}M$-connection. We have thus proved the following

\begin{proposition}\label{prop:J^1E_hol}
Formula (\ref{eq:flat_conn_J^1}) defines a canonical flat $T^{0,1}M$-connection $\overline \partial{}^{\mathfrak J^1 E}$ in $\mathfrak J^1 E$ (in particular, the bundle of first holomorphic jets is a holomorphic vector bundle over $X$).
\end{proposition}

\begin{remark}
Recall the direct sum decomposition $\mathfrak J^1_\R E = \mathfrak J^1 E \oplus (T^{0,1} M \otimes E)$. It is easy to see that the two complex structures on $\mathfrak J^1_\R E$ agree on the first summand and have opposite signs on the second summand. 
\end{remark}

\subsection{Multidifferential calculus on a holomorphic line bundle}

In this section we discuss the Schouten-Jacobi algebra of a holomorphic line bundle. It is analogue to the Schouten-Nijenhuis algebra of multivector fields on a manifold. The role of vector fields is now played by sections of the gauge algebroid. As the ``integrability condition'' of a Poisson bi-vector can be expressed in terms of the Schouten-Nijenhuis bracket, the ``integrability condition'' of a Jacobi structure can be expressed in terms of the Schouten-Jacobi bracket.  For details on the Schouten-Jacobi algebra of a real line bundle and its role in Jacobi geometry we refer to \cite[Appendix B]{LOTV2014}. Let $L \to M$ be a complex line bundle.

\begin{definition}
A  complex \emph {$k$-multiderivation} of $L$ is a $\C$-multilinear operator $\Gamma (L) \times \cdots \times \Gamma (L) \to \Gamma (L)$ which is skew-symmetric and  defines a first-order differential operator in each entry. Complex $k$-multiderivations of $L$ are sections of a complex vector bundle, denoted by $D^k_\C L \to M$. We write: $D^\bullet_\C L = \bigoplus_k D^k_\C L$. An element in $\Gamma (D^\bullet_\C L)$ is simply called a \emph{multiderivation}.
\end{definition}

\begin{remark}
Clearly $D^1_\C L = D_\C L \cong \operatorname{Hom} (\mathfrak J^1_\R L, L) 
$, where the homomorphisms are taken over $\C$ and $\mathfrak J^1_\R L$ is equipped with the fiber-wise complex structure given by $i \cdot \mathfrak j^1_\R \lambda := \mathfrak j^1_\R (i \lambda)$. More generally
\begin{equation}\label{eq:D^k_jets}
D^k_\C L \cong  \operatorname{Hom} (\wedge^k \mathfrak J^1_\R L, L),
\end{equation}
where the $k$-multiderivation $\Delta$ corresponds to the homomorphism $\Phi_\Delta : \wedge^k \mathfrak J^1_\R L \to L$ uniquely determined by 
\[
\Phi_\Delta (\mathfrak j^1_\R \lambda_1, \ldots, \mathfrak j^1_\R \lambda_k) := \Delta (\lambda_1, \ldots, \lambda_k).
\] 
\end{remark}

\begin{remark}
Since the rank of the complex vector bundle $L \to M$ plays no role in the definition of multiderivations, the latter is valid for any complex vector bundle.
\end{remark}

There is a degree zero graded Lie bracket $[ -,- ]^{SJ}$ on the graded vector space $\Gamma(D^\bullet_\C L)[1]$ ($\Gamma(D^\bullet_\C L)$ shifted by $1$) given by
\[
[ \Delta_1 , \Delta_2 ]^{SJ} := (-)^{k_1 k_2}\Delta_1 \circ \Delta_2 - \Delta_2 \circ \Delta_1,
\]
for all $\Delta_i \in \Gamma(D_\C^{k_i+1} L)$, $i = 1, 2$, where $\Delta_1 \circ \Delta_2$ is given by the following ``\emph{Gerstenhaber formula}'':
\[
\begin{aligned}
& (\Delta_1 \circ \Delta_2) (\lambda_1, \ldots, \lambda_{k_1+k_2 +1}) \\
& = \sum_{\tau \in S_{k_2+1, k_1}}(-)^{\tau} \Delta_1 (\Delta_2 (\lambda_{\tau(1)},\ldots, \lambda_{\tau(k_2+1)}) ,\lambda_{\tau (k_2+2)}, \ldots, \lambda_{\tau (k_1+k_2 +1)}),
\end{aligned}
\]
for all $\lambda_1,\ldots, \lambda_{k_1+k_2 +1} \in \Gamma (L)$. Here $S_{k,l}$ denotes \emph{unshuffles}. The bracket $[-,-]^{SJ}$ is called the \emph{Schouten-Jacobi bracket} (see \cite[Appendix B]{LOTV2014} for more details).

In what follows we denote by $\mathbb C_M := M \times \mathbb C \to M$ the trivial complex line bundle over $M$. When $L \to X$ is a holomorphic line bundle over a complex manifold $X = (M,j)$, then
the exact diagram 
\[
\begin{array}{c}
\xymatrix{  & 0 & 0 & 0 & \\
0 \ar[r] &  \C_M \ar[r] \ar[u] \ar@{=}[d]  &  D^{1,0} L \ar[r]^-{\sigma} \ar[u]  \ar@/^/[d]&  T^{1,0}M \ar[r]  \ar[u] &  0 \\
0 \ar[r]  & \C_M \ar[r]  & D_\C L \ar[r]^-{\sigma} \ar[u] &  (TM)^\C \ar[r] \ar[u] &  0 \\
            &  0 \ar[r]    \ar[u]    & T^{0,1} M \ar@{=}[r] \ar[u]^-{\overline \partial_{(-)}} & T^{0,1} M \ar[r] \ar[u] & 0 \\
            &  & 0 \ar[u] & 0 \ar[u] &
}
\end{array}
\]
is obtained from the exact diagram
\[
\begin{array}{c}
\xymatrix{ & 0 \ar[d] & 0 \ar[d] & 0 \ar[d] & \\
0  &  L \ar[l] \ar@{=}[d]  &  \mathfrak J^1 L \ar[l] \ar[d]  &  (T^{1,0}M)^\ast \otimes L \ar[l]_-{\gamma^{1,0}}  \ar[d] &  0 \ar[l] \\
0  & L \ar[l] \ar[d] & \mathfrak J_\R^1 L \ar[l] \ar[d]_-{\Phi_{\overline \partial}} \ar@/_/[u]&  TM \otimes_\R L \ar[l]_-{\gamma}  \ar[d] &  0 \ar[l] \\
    &    0       & (T^{0,1} M)^\ast \otimes L \ar[l] \ar[d] & (T^{0,1} M)^\ast \otimes L \ar@{=}[l] \ar[d] & 0 \ar[l] \\
     &  & 0 & 0 & }
\end{array}
\]
applying the functor $\operatorname{Hom} (-, L)$ (here we used $\operatorname{End}_{\C} L \cong \C_M$). The splitting $\mathfrak J^1_\R L = \mathfrak J^1 L \oplus (T^{0,1}M)^\ast \otimes L$ (see (\ref{eq:split_jet})), and the isomorphism (\ref{eq:D^k_jets}), now determine a factorization
\begin{equation}\label{eq:fact_D}
D_\C^\bullet L = D^{\bullet, 0} L \otimes T^{0,\bullet} M,
\end{equation}
where $D^{k, 0} L = \wedge^k (\mathfrak J^1 L)^\ast \otimes L$, and $D^{\bullet, 0} L := \bigoplus_k D^{k, 0} L$. The factorization (\ref{eq:fact_D}) extends the splitting $D_\C L = D^{1,0} L \oplus T^{0,1} M$ (see (\ref{eq:split_D})). Finally, denote by $\Delta \mapsto \Delta^{k,0}$ the projection $D_\C L \to D^{k, 0} L$. The vector bundle $D^{k,0} L$ is a holomorphic vector bundle over $X$ with a flat $T^{0,1}M$-connection also denoted by $\overline \partial{}^{DE}$ and given by
\[
\overline \partial{}^{DE}_\xi \Delta := \left([\overline \partial_\xi, \Delta]^{SJ} \right)^{k,0},
\]
for all $\xi \in \Gamma (T^{0,1}M)$, and $\Delta \in \Gamma (D^{k,0} L)$.

\section{Holomorphic Jacobi structures}\label{Sec:HolomJacobi}
\subsection{Holomorphic Jacobi manifolds}
Holomorphic Jacobi manifolds are the main objects in this paper. Before giving their full definition, we  recall that a (real) \emph{Jacobi manifold} is a triple $(M, L, \{-,-\})$, where $M$ is a manifold, $L \to M$ is a line bundle, and $\{-,-\} : \Gamma (L) \times \Gamma (L) \to \Gamma(L)$ is a (skew-symmetric) bi-derivation satisfying the Jacobi identity. The bracket $\{-,-\}$ is also called a \emph{Jacobi bracket} on $L \to M$, and the pair $(L, \{-,-\})$ is called a \emph{Jacobi bundle}.

\begin{example}\label{ex:real_cont}
Let $M$ be an odd dimensional manifold and let $\Cc \subset TM$ be a contact structure on it. Then the normal line bundle $L := TM/\Cc \to M$ is canonically equipped with a Jacobi bracket. See, e.g., \cite{CS2015} for details (see also \cite[Section 3]{V2015}).
\end{example}

\begin{definition}\label{def:main}
A \emph{holomorphic Jacobi manifold} is a complex manifold $X = (M,j)$ equipped with a \emph{holomorphic Jacobi structure}, i.e.~a pair $(L, J)$, where $L \to X$ is a holomorphic line bundle over $X$ and $J$ is a \emph{holomorphic Jacobi bi-derivation} of $L$, i.e.~a bi-derivation $J \in \Gamma (D^{2,0}L)$ such that 1) $\overline{\partial}{}^{DL} J = 0$, and 2) $[J, J]^{SJ} = 0$ (where $[-,-]^{SJ}$ is the \emph{Schouten-Jacobi bracket} of complex multiderivations of $L$). The pair $(L,J)$ is also called a \emph{holomorphic Jacobi bundle} over $X$.
\end{definition}

 There is a more algebraic definition of a Jacobi manifold (see Lemma \ref{Lem:JacMan} below). To see this, first recall that a \emph{Jacobi algebra} is a commutative algebra with unit equipped with a Lie bracket which is a first order diffferential operator in each entry. We here propose a slightly more general

\begin{definition} 
A \emph{Jacobi module} over a commutative algebra with unit $\mathcal A$ (or a \emph{Jacobi $\Aa$-module}), is an $\mathcal A$-module $\mathcal L$, equipped with
\begin{enumerate}
\item a Lie bracket $\mathcal L \times \mathcal L \to \mathcal L$, written $(\lambda_1, \lambda_2) \mapsto \{ \lambda_1, \lambda_2 \}$, and 
\item a Lie algebra homomorphism $\mathcal L \to \operatorname{Der} \mathcal A$, written $\lambda \mapsto R_\lambda$, such that
\end{enumerate}
\[
 \{ \lambda_1, a \lambda_2 \} = R_{\lambda_1} (a) \lambda_2 + a \{ \lambda_1, \lambda_2 \},
\]
for all $\lambda_1, \lambda_2 \in \mathcal L$ and $a \in \mathcal A$.
\end{definition}

\begin{lemma}\label{Lem:JacMan}
A holomorphic Jacobi manifold is the same as a holomorphic line bundle $L \to X$ equipped with the structure of a Jacobi $\Oo_X$-module on its sheaf $\Gamma_L$ of holomorphic sections, i.e.
\begin{enumerate}
\item for all open subsets $U \subset X$, a structure of Jacobi $\mathcal O_X(U)$-module on $\Gamma_L (U)$, such that
\item both the Lie bracket $\{-,-\} : \Gamma_L (U) \times \Gamma_L(U) \to \Gamma_L (U)$ and the Lie algebra homomorphisms $R : \Gamma_L (U) \to \operatorname{Der} \mathcal O_X(U)$ are compatible with restrictions.
\end{enumerate}
\end{lemma}

\begin{proof}
Suppose $X=(M, j)$ is a complex manifold and $(L, J)$ is a holomorphic Jacobi structure on it. Restricting $J : \Gamma (L) \times \Gamma (L) \to \Gamma (L)$ to holomorphic sections we get a Jacobi module structure on $\Gamma_L$. Conversely, suppose $L \to X$ is a holomorphic line bundle equipped with the structure of a Jacobi $\Oo_X$-module on its sheaf $\Gamma_L$ of holomorphic sections. First of all $R$ can be uniquely extended to a $\C$-linear, first order differential operator, also denoted $R$, from $\Gamma (L)$ to $\Gamma(T^{1,0}M)$, as follows. For $\lambda \in \Gamma (L)$ first define $R_\lambda$ (locally) on holomorphic functions. Thus, let $U \subset X$ be an open ball (i.e.~an open subset bi-holomorphic to an open ball in $\mathbb C^n$), let $\mu$ be a holomorphic generator of $\Gamma (L|_U)$, and let $f \in \mathcal O_X (U)$. A section $\lambda \in \Gamma (L|_U)$ can be always written as $\lambda = g \mu$ with $g \in C^\infty (M, \mathbb C)$. Write
\[
R_{\lambda} (f) := f R_\mu (g)+ g R_\mu (f) - R_{f \mu} (g).
\]
By construction, $R_\lambda$ is a well defined derivation of $\mathcal O_X (U)$ (with values in $C^\infty (U, \mathbb C)$), hence it extends uniquely to a (non-necessarily holomorphic) vector field, also denoted by $R_\lambda$, in $\Gamma(T^{1,0}U)$. Finally, define $\{-,-\} : \Gamma (L|_U) \times \Gamma (L|_U) \to \Gamma (L|_U)$ by
\[
\{ f \mu , g \mu \} = R_{f \mu} (g) \mu - g R_\mu (f) \mu,
\]
for all $f, g \in C^\infty (U, \mathbb C)$.
It is clear that the local data defined in this way, define a global $J$ such that $(X,L,J)$ is a holomorphic Jacobi manifold.
\end{proof}

\begin{example}[holomorphic contact manifolds]\label{ex:cc}
Recall that a \emph{holomorphic contact structure} on a complex manifold $X = (M,j)$ of (complex dimension) $2n+1$ is given by a holomorphic vector subbundle $\Cc \subset T^{1,0}M$ of rank $2n$ which is completely non-integrable in the sense that the \emph{Frobenius map}:
\[
\wedge^2 \Cc \longrightarrow T^{1,0} M / \Cc, \quad (\xi, \eta) \longmapsto [\xi, \eta] \operatorname{mod} \Cc ,
\]
is everywhere nondegenerate. Now, let $(X, \Cc)$ be a holomorphic contact manifold, i.e.~a complex manifold equipped with a holomorphic contact structure, and consider the holomorphic line bundle $L = T^{1,0} M/\Cc$, called the \emph{contact line bundle} in the sequel, along with the exact sequence:
\[
0 \longrightarrow \Cc \longrightarrow T^{1,0}M \overset{\theta_\Cc}{\longrightarrow} L \longrightarrow 0,
\]
where $\theta_\Cc$ is the canonical projection. In fact, $\theta_\Cc$ could be viewed as a holomorphic $1$-form on $X$ with values in $L$. When $L \cong \C_X$, then $\theta_\Cc$ can be viewed as a standard holomorphic $1$-form on $X$, called a (\emph{global}) \emph{contact} $1$-form. Contact $1$-forms do always exist locally. 

The contact line bundle $L \to X$ of a holomorphic contact manifold $(X, \Cc)$ is naturally equipped with a holomorphic Jacobi bracket that can be constructed in the same way as in the real case of Example \ref{ex:real_cont} (see, e.g., \cite{CS2015}, see also \cite[Section 3]{V2015}). Even more, holomorphic contact structures with fixed contact line bundle $L \to X$, are in one-to-one correspondence with \emph{non-degenerate} holomorphic Jacobi bi-derivations of $L$ \cite[Section 3]{V2015}.
\end{example}

\begin{example}
Let $A \to X$ be a holomorphic Lie algebroid over a complex manifold $X = (M,j)$, and let $L \to X$ be a holomorphic line bundle equipped with a flat holomorphic $A$-connection. Then the complex manifold $A^\ast \otimes L$ is canonically equipped with a holomorphic Jacobi bundle $\mathrm{pr}^\ast L \to A^\ast \otimes L$, where $\mathrm{pr} : A^\ast \otimes L \to X$ is the projection. The Jacobi bracket between holomorphic sections of $\mathrm{pr}^\ast L \to A^\ast \otimes L$ can be defined in the same way as in the real case \cite[Subsection 2.3]{LOTV2014}. Actually, given a holomorphic complex bundle $A \to X$ and a holomorphic line bundle $L \to X$, the following sets of data are equivalent:
\begin{itemize}
\item a holomorphic Lie algebroid structure on $A \to X$ and a flat holomorphic $A$-connection in $L$;
\item a \emph{fiber-wise linear} holomorphic Jacobi bundle structure on $\mathrm{pr}^\ast A \to A^\ast \otimes L$
\end{itemize}
(see \cite[Subsection 2.3]{LOTV2014} for the notion of \emph{fiber-wise linear Jacobi brackets}).
\end{example}

\subsection{Jacobi-Nijenhuis manifolds and generalized contact bundles}

In Section \ref{sec:HJ_JN_GC} we  will show that holomorphic Jacobi manifolds are equivalent to
\begin{enumerate}
\item (real) Jacobi-Nijenhuis manifolds,
\item generalized contact bundles,
\end{enumerate}
of a certain kind. First we will introduce here these notions  (1) and (2). Notice that Jacobi-Nijenhuis manifolds were studied  in \cite{MMP1999}. But here we take a slightly more general point of view, where the Jacobi structure lives on a non-necessarily trivial line bundle.

First of all, we fix our notation. Let $\{-,-\} : \Gamma (L) \times \Gamma(L) \to \Gamma (L)$
be a bi-derivation of a real line bundle $L \to M$. In view of the isomorphism (\ref{eq:D^k_jets}) we can regard $\{-,-\}$ as a vector bundle
morphism $\wedge^2 \mathfrak J^1_\R L \to L$ (\ref{eq:D^k_jets}). When doing this we
will use the symbol $J$. In sum, from now on, unless otherwise stated, $\{-,- \}$
will denote a skew-symmetric bracket on $\Gamma (L)$ which is a derivation in each
entry, while $J$ will denote the corresponding $L$-valued skew-symmetric bilinear form
on $\mathfrak J^1_\R L$, so that $\{-,-\}$ and $J$ contain the same information. To express this fact we also write $J \equiv \{-,-\}$ (or $\{-,-\} \equiv J$). The $2$-form
$J$ determines an obvious vector bundle morphism $J^\sharp: \mathfrak J^1_\R L \to D_\R L \cong \operatorname{Hom}_\R (\mathfrak J^1_\R L, L)$. Notice that, for any
$\Delta \in \Gamma (D_\R L)$ there is a unique derivation
$\Ll_\Delta : \Gamma (\mathfrak J^1_\R L) \to \Gamma (\mathfrak J^1_\R L)$,
such that 1) $\Delta$ and $\Ll_\Delta$ share the same symbol, and
2) $\Ll_\Delta \mathfrak j^1_\R \lambda = \mathfrak j^1_\R \Delta \lambda$
for all $\lambda \in \Gamma(L)$. The derivation $\Ll_\Delta$ is the \emph{Lie derivative along $\Delta$}.
We are now ready to see that a bi-derivation $\{-,-\} \equiv J$ of $L$ determines a skew-symmetric bracket $[-,-]_J$ on $\Gamma (\mathfrak J^1_\R L)$ given by
\begin{equation}\label{eq:[-,-]_J}
[\rho, \tau]_J := \Ll_{J^\sharp \rho} \tau - \Ll_{J^\sharp \tau} \rho - \mathfrak j^1_\R J (\rho, \tau),
\end{equation}
for all $\rho, \tau \in \Gamma (\mathfrak J^1_\R L)$. A direct computation shows that $J$ is a Jacobi bi-derivation if and only if $[-,-]_J$ is a Lie bracket. In this case $[-,-]_J$ is the Lie bracket on sections of the \emph{jet algebroid} $(\mathfrak J^1_\R L)_J$ of the Jacobi manifold $(M, L, J)$ \cite{LOTV2014}. Now, let $\phi : D_\R L \to D_\R L$ be an endomorphism, and let $\phi^\dag : \mathfrak J^1_\R L \to \mathfrak J^1_\R L$ be its adjoint, i.e.~$\langle \phi^\dag (\rho), \Delta \rangle := \langle \rho, \phi (\Delta) \rangle$, for all $\rho \in \mathfrak J^1_\R L$, and $\Delta \in D_\R L$, where $\langle -, -\rangle : \mathfrak J^1_\R L \times D_\R L \to L$ is the ``duality pairing''. If $J^\sharp \circ \phi^\dag = \phi \circ J^\sharp$, then $J_\phi := J (\phi -,-)$ is a well-defined $L$-valued skew-symmetric form on $\mathfrak J^1_\R L$ such that $J^\sharp_\phi = J^\sharp \circ \phi^\dag$. Denote by $\{-,-\}_\phi$ the bi-derivation corresponding to $J_\phi$: $J_\phi \equiv \{-,-\}_\phi$.

Now, let $\{-,-\}$ bi a bi-derivation of $L$, let $J : \wedge^2 \mathfrak J^1_\R L \to L$ be the associated $2$-form, and let $\phi : D_\R L \to D_\R L$ be an endomorphism. We say that $\phi$ is \emph{compatible} with $J$ if
\begin{equation}
J^\sharp \circ \phi^\dag = \phi \circ J^\sharp,
\end{equation}
(hence $J_\phi$ is well-defined) and 
\begin{equation}
\phi^\dag [\rho, \tau]_J = [\phi^\dag \rho, \tau]_J + [\rho, \phi^\dag \tau]_J - [\rho, \tau]_{J_\phi},
\end{equation}
for all $\rho, \tau \in \Gamma (J^1_\R L)$.

\begin{definition}
A \emph{Jacobi-Nijenhuis manifold} is a manifold $M$ equipped with a \emph{Jacobi-Nijenhuis structure}, i.e.~a triple $(L, \{-,-\}, \phi)$, where $L \to M$ is a line bundle, $\{-,-\}$ is a Jacobi bracket on $L$, and  $\phi : D_\R L \to D_\R L$ is a compatible endomorphism whose \emph{Nijenhuis torsion} $\mathcal N_\phi : \wedge^2 D_\R L \to D_\R L$ vanishes identically. Here
\[
\mathcal N_\phi (\Delta_1,\Delta_2) :=  [\phi (\Delta_1), \phi (\Delta_2)] + \phi^2 [\Delta_1,\Delta_2] - \phi [\phi (\Delta_1), \Delta_2] - \phi [\Delta_1, \phi (\Delta_2)],
\]
for all $\Delta_1,\Delta_2 \in \Gamma(D_\R L)$, .
\end{definition}

\begin{proposition}\label{prop:JbH}
Let $(L, \{-,-\} \equiv J, \phi)$ be a Jacobi-Nijenhuis structure. Then $(L, \{-,-\}, \{-,-\}_\phi)$ is a \emph{Jacobi bi-Hamiltonian} structure, i.e.~$\{-,-\}$, $\{-,-\}_\phi$ and $\{-,-\} + \{-,-\}_\phi$ are all Jacobi brackets.
\end{proposition}

We now recall the definition of a generalized contact bundle from \cite{VW2016}. Let $L \to M$ be a line bundle. The \emph{omni-Lie algebroid} \cite{CL2010, CLS2011} $\D L := D_\R L \oplus \mathfrak J^1_\R L$ is canonically equipped with the following structures:
\begin{itemize}
\item the projection $\operatorname{pr}_D : \D L \to D_\R L$, 
\item the symmetric bilinear form $\langle \hspace{-2.7pt} \langle -,- \rangle \hspace{-2.7pt} \rangle : \D L \otimes \D L \to \mathbb L$:
\[
\langle \hspace{-2.7pt} \langle (\Delta, \rho), (\nabla, \tau) \rangle \hspace{-2.7pt} \rangle := \langle \tau , \Delta \rangle + \langle \rho, \nabla \rangle,
\]
\item the \emph{Dorfman-Jacobi bracket} $[\![ -,-]\!] : \Gamma (\D L) \times \Gamma (\D L) \to \Gamma (\D L)$:
\[
[\![  (\Delta, \rho), (\nabla, \tau)]\!] := ([\Delta,\nabla], \Ll_{\Delta} \tau - \Ll_{\nabla} \rho + \mathfrak j_\R^1 \langle \rho, \nabla \rangle),
\]
\end{itemize}
where $\Delta, \nabla \in \Gamma (D_\R L)$, and $\rho, \tau \in \Omega^1 (M)$. With the above three structures $\D L$ is an \emph{$L$-Courant algebroid} \cite{CLS2010} and a \emph{contact Courant algebroid} \cite{G2013}.

\begin{definition}[\cite{VW2016}]
A \emph{generalized contact bundle} is a line bundle $L \to M$ equipped with a \emph{generalized contact structure}, i.e.~a vector bundle endomorphism $\Ii : \D L \to \D L$ such that
\begin{itemize}
\item $\Ii$ is \emph{almost complex}, i.e.~$\Ii^2 = -1$,
\item $\Ii$ is \emph{skew-symmetric}, i.e.
\[
\langle \hspace{-2.7pt} \langle \Ii \alpha, \beta \rangle \hspace{-2.7pt} \rangle + \langle \hspace{-2.7pt} \langle \alpha, \Ii \beta \rangle \hspace{-2.7pt} \rangle = 0, \quad \text{for all }\alpha, \beta \in \Gamma (\D L),
\]
\item $\Ii$ is \emph{integrable}, i.e.
\[
[\![ \Ii \alpha, \Ii \beta]\!] - [\![ \alpha, \beta]\!] - \Ii [\![ \Ii \alpha, \beta]\!] + \Ii [\![ \alpha, \Ii \beta ]\!] = 0, \quad \text{for all }\alpha, \beta \in \Gamma (\D L).
\]
\end{itemize}
\end{definition} 
Let $(L \to M, \Ii)$ be a generalized contact bundle. Using the direct sum decomposition $\D L = D_\R L \oplus \mathfrak J^1_\R L$, and the definition, one can see that
\[
\Ii = \left(
\begin{array}{cc}
\phi & J^\sharp \\
\omega_\flat & - \phi^\dag
\end{array}
\right)
\]
where $J$ is a Jacobi bi-derivation, $\phi : D_\R L \to D_\R L$ is an endomorphism compatible with $J$, and $\omega : \wedge^2 D_\R L \to L$ is a $2$-form, with associated vector bundle morphism $\omega_\flat : D_\R L \to \mathfrak J^1_\R L$, satisfying additional compatibility conditions \cite{VW2016}. In particular, when $\omega = 0$, then $\phi$ is a complex structure in the vector bundle $D_\R L \to M$, and $(L, J, \phi)$ is a Jacobi-Nijenhuis structure. Generalized contact bundles are supported by odd dimensional manifolds and their geometry is an odd dimensional analogue of generalized complex geometry. 

\subsection{Jacobi-Nijenhuis and homogeneous Poisson-Nijenhuis, generalized contact and homogeneous generalized complex}

Jacobi manifolds are in one-to-one correspondence with certain homogeneous 
Poisson manifolds.  The correspondence  is provided by the 
\emph{Poissonization construction} in one direction, and by Proposition 
\ref{prop:J_hP} below in the other direction. Similarly, there are correspondences 
between Jacobi-Nijenhuis manifolds (resp.~generalized contact bundles) and 
homogeneous Poisson-Nijenhuis manifolds (resp.~homogeneous generalized 
complex manifolds), see Theorem \ref{theor:JN_hPN} (resp.~Theorem 
\ref{theor:gc_hgc}). Indeed, let $L \to M$ be a (real) line bundle. The manifold 
$\widetilde M := L^\ast \smallsetminus \{0\}$ is a principal bundle over $M$ with
 structure group the multiplicative group $\R^\times := \R \smallsetminus \{0\}$,
  and we denote by $h : \R^\times \times \widetilde M \to \widetilde M$, $(r, \epsilon) \mapsto h_r (\epsilon)$ the principal action. 
  Denote by $p : \widetilde M \to M$ the projection. Let $\eta$ be 
  the restriction to $\widetilde M$ of the \emph{Euler vector field} 
  on $L^\ast$. Then $\eta$ is the fundamental vector field corresponding
   to the canonical generator $1$ in the Lie algebra $\R$ of the structure 
   group $\R^\times$. Sections of $L$ are in one-to-one correspondence 
   with \emph{homogeneous functions} on $\widetilde M$, 
  i.e.~functions $f \in C^\infty (\widetilde M)$ such that $h_r^\ast (f) = r f$ for every $r \in \R^\times$. In their turn homogeneous functions can be 
   characterized as those functions $f$ such that $\Ll_\eta f = f$, and, 
   additionally, $h_{-1} (f) = -f$. Notice that, as the Lie group $\R^\times$ 
   is not connected, the last condition \emph{cannot} be removed.
For any section $\lambda \in \Gamma (L)$, we denote by $\widetilde \lambda$ the corresponding homogeneous function on $\widetilde M$. Now, let $\{-,-\} : \Gamma (L) \times \Gamma(L) \to \Gamma (L)$ be a bi-derivation, i.e.~a skew-symmetric bracket which is a first order differential operator in each argument. It is easy to see that there is a unique bi-vector $\pi$ on $\widetilde M$ such that
\[
\{ \widetilde \lambda_1, \widetilde \lambda_2 \}_{\widetilde{M}}  = \widetilde{\{ \lambda_1, \lambda_2 \} },
\]  
for all $\lambda_1, \lambda_2 \in \Gamma(L)$, where $\{-,-\}_{\widetilde{M}} : C^\infty (\widetilde M) \times C^\infty (\widetilde M) \to C^\infty (\widetilde M)$ is the skew-symmetric bracket determined by $\pi$, i.e.~$\{f_1,f_2\}_{\widetilde M} = \pi (df_1, df_2)$. Additionally, $\Ll_\eta \pi = -\pi$. Finally, $\pi$ is a \emph{Poisson bi-vector}, hence $(\pi, \eta)$ is a homogeneous Poisson structure on $\widetilde M$, if and only if $\{-,-\}$ is a Jacobi bracket, i.e.~it satisfies the Jacobi identity. In this case, the homogeneous Poisson manifold $(\widetilde M, \pi, \eta)$ is sometimes called the \emph{Poissonization} of the Jacobi manifold $(M,L, \{-,-\})$. Notice that $\pi$ comes from a symplectic structure if and only if $J$ comes from a contact structure. 

\begin{proposition}\label{prop:J_hP} %\cite{DLM1991}
Every homogeneous Poisson manifold is the Poissonization of a canonical Jacobi manifold around every non-singular point of the homogeneity vector field.
\end{proposition}

\begin{proof}
Let $(N, \pi_N, \eta_N)$ be an $(n+1)$-dimensional homogeneous Poisson manifold, and let $x \in N$ be a point such that $(\eta_N)_x \neq 0$. Then $\eta_N$ is everywhere non-zero in a whole neighborhood $U$ of $x$. Even more, $U$ can be chosen so that there is a diffeomorphism $F : U \to (a,b) \times M$, where $(a,b)$ is an interval with $0 < a < b$, $M$ in an $n$-dimensional manifold, and $F$ intertwines $\eta_N$ and $t \frac{d}{dt}$, $t$ being the canonical coordinate on $(a,b)$. In the following, we use $F$ to identify $U$ with $(a,b) \times M$, and $\eta_N$ with $t \frac{d}{dt}$. Then, the space of orbits of $\eta_N$ in $U$ is simply $M$, and the projection $p : U = (a,b) \times M \to M$ is the projection onto the second factor. Smooth functions $f$ on $U$ that are homogeneous with respect to $\eta_N$, i.e.~such that $\Ll_{\eta_N} f = f$, are linear functions in the coordinate $t$, and they form a $C^\infty (M)$-module $\Ll$. The value of any such function on a fiber $\mathcal F$ of $p$ is completely determined by its value on a point of $\mathcal F$. Hence $\Ll$ is the module of sections of a line bundle $L \to M$. Since $\pi_N$ is homogeneous with respect to $\eta_N$, then the corresponding Poisson bracket restricts to a Jacobi bracket $\{-,-\}$ on $L$. Our final aim is to show that $U$ can be embedded in $\widetilde M = L^\ast \smallsetminus \{0\}$ as an an \emph{open homogeneous Poisson submanifold}. Clearly, there is a unique smooth Poisson map $i : U \to \widetilde M$ intertwining $\eta_N$ and $\eta$, such that 
\[
\widetilde \lambda (i(x)) = \lambda (x), 
\]
for all $x \in U$ and $\lambda \in \Gamma (L) = \Ll \subset C^\infty (U)$. Since $\Ll$ generates $C^\infty (U)$ as a $C^\infty$-algebra, then $i$ is the inclusion of an open (homogeneous Poisson) submanifold. 
\end{proof}

\begin{remark}
 Proposition \ref{prop:J_hP} is not completely new and should be compared with a similar one, namely, Proposition 2.3 in \cite{DLM1991}. There are two main differences, however, between our approach and that of \cite{DLM1991}. First of all, the authors of \emph{loc. cit.} work with trivial line bundles, Jacobi pairs and conformal Jacobi maps, which makes things slightly less canonical. This leads to the second difference: in order to construct a Jacobi manifold from a homogeneous Poisson manifold, the authors of \cite{DLM1991} have to choose a hypersurface transverse to the homogeneity vector field (at a non-singular point), and then show that the construction is independent of the choice of the hypersurface up to conformal Jacobi isomorphisms. In our case, working with non-necessarily trivial line bundles and Jacobi structures on them, we can simply take a quotient (see the proof of Proposition \ref{prop:J_hP}) which doesn't require any choice. 
 \end{remark}
 
% \begin{remark}\label{rem:not_Jac}
% Notice that we can work with any pair $(\pi_N, \eta_N)$ such that $\pi_N$ is a (not necessarily Poisson) bi-vector on $N$ and $\eta_N$ is a vector field on $N$ such that $\Ll_{\eta_N} \pi_N = - \pi_N$. In this case, $\{-,-\}$ is still a well-defined bi-derivation of $L \to M$, and it satisfies the Jacobi identity if and only if $\pi_N$ is a Poisson bi-vector.
%\end{remark}

\begin{theorem}\label{theor:JN_hPN}
The Poissonization of a Jacobi-Nijenhuis manifold is a homogeneous Poisson-Nijenhuis manifold in a natural way. Conversely, every homogeneous Poisson-Nijenhuis manifold is the Poissonization of a canonical Jacobi-Nijenhuis manifold around every non-singular point of the homogeneity vector field.
\end{theorem}

\begin{proof}
Let $L \to M$ be a line bundle and let $\{-,-\}$ be a bi-derivation of $L$. Consider the slit dual $\widetilde M = L^\ast \smallsetminus \{0\}$, the Euler vector field $\eta \in \mathfrak X (\widetilde M)$ and the bi-vector $\pi$ determined by $\{-,-\}$ on $\widetilde M$ as in the beginning of this section (if, in particular, $(M, L, \{-,-\})$ is a Jacobi manifold, then $(M, \pi, \eta)$ is its Poissonization). Finally let $\phi : D_\R L \to D_\R L$ be an endomorphism. From \cite[Proposition A.1]{V2015} there is an embedding
$\iota : \Gamma (D_\R L) \to \mathfrak X (\widetilde M)$, $\Delta \mapsto \widetilde \Delta$
of $C^\infty (M)$-modules, and Lie algebras, uniquely determined by
$\widetilde \Delta (\widetilde \lambda) = \widetilde{\Delta (\lambda)}$,
for all $\lambda \in \Gamma(L)$, and there is a canonical isomorphism
$p^\ast D_\R L \cong T \widetilde M$ of vector bundles over $\widetilde M$
such that $\iota$ agrees with the pull-back along $p$ of sections of $DL$.
In particular, there is a unique (1,1) tensor $\widetilde \phi$ on $T \widetilde M$
such that $\widetilde \phi (\widetilde \Delta) = \widetilde{\phi (\Delta)}$ for all
$\Delta \in \Gamma (D_\R L)$. It is not hard to see that $(\widetilde M, \pi, \widetilde \phi, \eta)$
is a homogeneous Poisson-Nijenhuis manifold if and only if $(M, L, \{-,-\}, \phi)$ is a Jacobi-Nijenhuis manifold. We leave the details to the reader.

Conversely, let $(N, \pi_N, \phi_N, \eta_N)$ be a homogeneous Poisson-Nijenhuis manifold and let $p : U \to M$, $(M, L, \{-,-\})$, 
and $i = U \to \widetilde M$ be as in the proof of Proposition \ref{prop:J_hP}. From $\Ll_{\eta_N} \phi_N = 0$, it follows that $\phi_N$ restricts to an endomorphism $\phi : D_\R L \to D_\R L$. 
Moreover, $(M, L, J, \phi)$ is a Jacobi-Nijenhuis manifold. This concludes the proof.
\end{proof}

Now, let $(L \to M, \Ii)$ be a generalized contact bundle,
$\Ii =
\left( \begin{smallmatrix}
\phi & J^\sharp \\
\omega_\flat & -\phi^\dag
\end{smallmatrix} \right)$.
Consider $\widetilde M = L^\ast \smallsetminus \{0\}$ and the Euler vector field $\eta$ on it. Since $L$ is equipped with a Jacobi structure $J$, then $\widetilde M$ is canonically equipped with a homogeneous Poisson structure $\pi$. We call $(\widetilde M, \pi, \eta)$ the \emph{Poissonization} of $(L \to M, \Ii)$. 

\begin{theorem}\label{theor:gc_hgc}
The Poissonization of a generalized contact bundle is a homogeneous generalized complex manifold in a natural way. Conversely, every homogeneous generalized complex manifold is the Poissonization of a canonical generalized contact bundle around every non-singular point of the homogeneity vector field.
\end{theorem}

\begin{proof}
We discussed in \cite[Remark 3.6, arXiv version]{VW2016} that $\pi$ can be canonically completed to a homogeneous generalized complex structure $(\Jj, \eta)$ (see Definition \ref{def:hgc}):
$\Jj =
\left( \begin{smallmatrix}
\widetilde \phi & \pi^\sharp \\
\widetilde \omega_\flat & -\widetilde \phi{}^\ast
\end{smallmatrix} \right)$.
The second part of the statement can be proved along similar lines as in the proof of Theorem \ref{theor:JN_hPN}. We leave the details to the reader.
\end{proof}

\begin{remark}\label{rem:one-to-one}
Let $L \to M$ be a real line bundle and let $\widetilde M = L^\ast \smallsetminus \{0\}$ be its \emph{slit dual}. Denote by $\eta$ the Euler vector field on $\widetilde M$. We proved that, given a bi-derivation $J$ of $L$ and an endomorphism $\phi : D_\R L \to D_\R L$, we can canonically construct a bi-vector $\pi$ and a (1,1) tensor $\widetilde \phi$ on $\widetilde M$. This construction establishes a one-to-one correspondence between pairs $(J, \phi)$ consisting of a bi-derivation $J$ of $L$ and an endomorphism $\phi$ of $DL$, and pairs $(\pi, \widetilde \phi)$ consisting of a bi-vector and a (1,1) tensor on $\widetilde M$ such that
$h_r^\ast (\pi) = r^{-1} \pi$, and $h_r^\ast (\phi) = \phi$ for all $r \in \R^\times$, or, equivalently, $\Ll_\eta \pi = -\pi$, $h_{-1}^\ast (\pi) = -\pi$, and $\Ll_\eta \widetilde \phi = 0$, $h_{-1}^\ast (\phi) = \phi$. Finally, this correspondence identifies Jacobi-Nijenhuis structure $(L, J, \phi)$ on $M$, and homogeneous Poisson-Nijenhuis structures $(\pi, \widetilde \phi, \eta)$ on $\widetilde M$, with the additional properties that $h_{-1}^\ast (\pi) = -\pi$, and $h_{-1}^\ast (\phi) = \phi$ (and similarly for generalized contact and homogeneous generalized complex structures). Notice again that the last condition is unavoidable as the Lie group $\R^\times$ is not connected.
\end{remark}

\subsection{Holomorphic Jacobi and homogeneous holomorphic Poisson manifolds}

In this section we show that, similarly as for real Jacobi manifolds, holomorphic Jacobi manifolds are just equivalent to certain homogeneous holomorphic Poisson manifolds. In their turn homogeneous holomorphic Poisson manifolds are related to homogeneous Poisson-Nijenhuis manifolds and to homogeneous generalized complex manifolds (Theorem \ref{theor:hhP_hPN_hgc}). Finally homogeneous Poisson-Nijenhuis manifolds are equivalent to certain Jacobi-Nijenhuis manifolds (Theorem \ref{theor:JN_hPN}), and homogeneous generalized complex manifolds are  equivalent to certain generalized contact bundles (Theorem \ref{theor:gc_hgc}). In this way, we depict the following picture
\[  
\xymatrix{ \text{Jacobi-Nijenhuis} \ar@{<=>}[d] & \text{holomorphic Jacobi} \ar@{=>}[r] \ar@{=>}[l] \ar@{<=>}[d] & \text{generalized contact} \ar@{<=>}[d] \\
\txt{homogeneous \\ Poisson-Nijenhuis} \ 
& \txt{homogeneous \\ holomorphic Poisson}\  \ar@{=>}[r] \ar@{=>}[l] & \txt{homogeneous \\ generalized complex}\ 
}
\]
However, there is a difference between the holomorphic Jacobi and the holomorphic Poisson case. Namely, every holomorphic Poisson manifold is itself a (real) Poisson-Nijenhuis manifold. On the other hand, a holomorphic Jacobi structure on a complex manifold $X = (M,j)$ gives rise to a Jacobi-Nijenhuis structure, but not on $M$, rather on a larger manifold, specifically, an appropriate circle bundle over $M$. This phenomenon makes holomorphic Jacobi manifolds particularly interesting and have consequences also on the associated integration problem (see \cite{VW2019}).

\subsubsection{From holomorphic Jacobi to homogeneous holomorphic Poisson manifolds}

Let $L \to X$ be a holomorphic line bundle over a complex manifold $X = (M,j)$ and let $L^\ast$ be its complex dual. In the following we denote by $\widetilde M$ the \emph{slit complex dual} of $L$, i.e. $\widetilde M := L^\ast \smallsetminus \{0 \}$ (beware that in the previous section, we denoted by the same symbol $\widetilde M$ a different object, namely the \emph{slit real dual} of a \emph{real} line bundle). The real manifold $\widetilde M$ is equipped with a complex structure $\widetilde j $ induced by the complex structure on $L$, so that $\widetilde X = (\widetilde M, \widetilde j)$ is a holomorphic principal bundle over $X$ with structure group the multiplicative $\C^\times := \C \smallsetminus \{0\}$. Denote by $p : \widetilde X \to X$ the projection. Let $H \in \Gamma (T^{1,0} \widetilde X)$ be the restriction to $\widetilde X$ of the holomorphic Euler vector field on $L^\ast$. Then $H$ is the fundamental vector field corresponding to the canonical generator $1$ in the complex Lie algebra $\C$ of the structure group $\C^\times$. It is easy to see that $H = \frac{1}{2} (\eta -i \widetilde j  \eta)$, where $\eta \in \mathfrak X (\widetilde M)$ is the (restriction of) the real Euler vector field of the real (rank $2$) vector bundle $L^\ast \to M$.

\begin{proposition}\label{prop:Poissonization}
A holomorphic Jacobi bracket $J \equiv \{-,-\}$ on $L \to X$ determines canonically a  homogeneous holomorphic Poisson structure $(\Pi, H)$ on $\widetilde X$. Additionally, $J$ comes from a holomorphic contact structure (see Example \ref{ex:cc}) if and only if $\Pi$ comes from a holomorphic symplectic structure.
\end{proposition}

\begin{proof}
The proof is similar to that of the analogous statement in the smooth category. We report a sketch here for completeness. First of all, denote by $h : \C^\times \times \widetilde X \to \widetilde X$, $(c, \epsilon) \mapsto h_c (\epsilon)$ the principal action, and notice that holomorphic sections of $L \to X$ are in one-to-one correspondence with homogeneous holomorphic functions on $\widetilde X$, i.e.~functions $f \in \Oo_{\widetilde X}$ such that $h_c^\ast (f) = cf$ for all $c \in \C^\times$. Equivalently, homogeneous holomorphic functions $f$ can be characterized by the simpler condition $\mathcal L_{H} f = f$ (unlike in the real case, we don't need further conditions because the Lie group $\C^\times$ is \emph{connected}). For any holomorphic section $\lambda$ of $L \to X$, we denote by $\widetilde \lambda$ the corresponding homogeneous holomorphic function on $\widetilde X$. Now, let $J \equiv \{-,-\} \in \Gamma (D^{2,0}_\C L)$ be a holomorphic bi-derivation of $L \to X$, i.e.~$\overline \partial{}^{DL} J = 0$. It is easy to see that there exists a unique holomorphic bivector $\Pi$ on $\widetilde X$ such that
\[
\{ \widetilde \lambda_1, \widetilde \lambda_2 \}_{\widetilde{X}}  = \widetilde{\{ \lambda_1, \lambda_2 \} },
\] 
for all holomorphic sections $ \lambda_1, \lambda_2$ of $L$, where $\{-,-\}_{\widetilde{X}} : \Oo_{\widetilde X} \times \Oo_{\widetilde X} \to \Oo_{\widetilde X}$ is the skew-symmetric bracket corresponding to $\pi$, i.e.~$\{f_1,f_2\}_{\widetilde{X}} = \Pi (\partial f_1, \partial f_2)$. Additionally, $\Ll_H \Pi = -\Pi$. Finally, $\Pi$ is a \emph{holomorphic Poisson bi-vector}, hence $(\Pi, H)$ is a  homogeneous holomorphic Poisson structure on $\widetilde X$, if and only if $J$ is a Jacobi bracket, i.e.~it satisfies the Jacobi identity.

The second part of the statement can be proved, e.g., in local coordinates. We leave the details to the reader. We only remark that, if $\mu$ is a holomorphic section which generates $\Gamma (L)$ locally in its domain, and $\Cc \subset T^{1,0}M$ is a holomorphic contact structure with $L = T^{1,0}M / \Cc$, then, in view of the \emph{holomorphic contact Darboux lemma} (see, e.g., \cite[Theorem A.2]{AFL2017}) there are complex coordinates $(t,z^i, P_i)$ on $X$ such that
\[
\theta_\Cc = \left( dt - P_k dz^k \right) \otimes \mu .
\]  
where $\theta_\Cc : T^{1,0}M \to L$ is the projection. It is now easy to see that the corresponding holomorphic (homogeneous) symplectic structure $\Omega$ on $\widetilde X$ is locally given by
\[
\Omega = d \widetilde \mu \wedge (dt - P_k dz^k) - \widetilde \mu \, dP_k \wedge dz^k.
\]
\end{proof}

\begin{remark}\label{rem:one-to-one_2}
Let $L \to X$ be a holomorphic line bundle and let $\widetilde X = L^\ast \smallsetminus \{0 \}$ be its \emph{slit complex dual}. Denote by $H$ the Euler vector field on $\widetilde X$. The proof of Proposition \ref{prop:Poissonization} shows that, given a holomorphic bi-derivation $J$ of $L$, we can canonically construct a holomorphic bi-vector $\Pi$ on $\widetilde X$. This construction establishes a one-to-one correspondence between holomorphic bi-derivations of $L$ and homogeneous holomorphic bi-vectors on $\widetilde X$. Additionally, this correspondence identifies holomorphic Jacobi structures $(L, J)$ on $X$, and  homogeneous holomorphic Poisson structures $(\Pi, H)$ on $\widetilde X$ (and we don't need further conditions because the Lie group $\C^\times$ is connected).
\end{remark}

\begin{example}
Let $\mathfrak g$ be a complex Lie algebra and let $(\Pi, H)$ be the  homogeneous holomorphic Poisson structure on its complex dual $\mathfrak g^\ast$ (see Example \ref{ex:complex_Lie_algebra}). We also consider the complex projective space $\C\P (\mathfrak g^\ast)$. Call it $X$, and denote by $L \to X$ the complex dual of the tautological (line) bundle over $\C\P (\mathfrak g^\ast)$. Clearly, $\widetilde X = \mathfrak g^\ast \smallsetminus \{0 \}$. Additionally $H$ identifies with the Euler vector field on $\widetilde X$. Together with Remark \ref{rem:one-to-one_2}, this shows that $L \to X$ is equipped with a canonical holomorphic Jacobi bundle structure.
\end{example}

\begin{example}
Let $X_0$ be a complex manifold and recall that the cotangent bundle $T^\ast X_0$ is equipped with a canonical homogeneous holomorphic symplectic structure $(\Omega, H)$ where $\Omega$ is the canonical holomorphic symplectic form and $H$ is the holomorphic Euler vector field (see Example \ref{ex:holomorphic_cotangent_bundle}). We also consider the complex projective bundle $\C\P(T^\ast X_0)$. Call it $X$, and denote by $L \to X$ the complex dual of the tautological (line) bundle over $\C\P(T^\ast X_0)$. Clearly, $\widetilde X = T^\ast X_0 \smallsetminus \{0\}$, and $H$ identifies with the Euler vector field on $\widetilde X$. Hence there is a canonical holomorphic contact structure on $\C\P(T^\ast X_0)$. The latter agrees with the standard holomorphic contact structure on the complex manifold of holomorphic contact elements in $X_0$.
\end{example}

\begin{example}
More generally, let $A \to X_0$ be a holomorphic Lie algebroid and let $(\Pi, H)$ be the  homogeneous holomorphic Poisson structure on its complex dual $A^\ast$ (see Example \ref{example:holomorphic_dual}). We also consider the complex projective bundle $\C\P (A^\ast)$. Call it $X$, and denote by $L \to X$ the complex dual of the tautological (line) bundle over $\C\P (A^\ast)$. Clearly, $\widetilde X = A^\ast \smallsetminus \{0\}$. Additionally $H$ identifies with the Euler vector field on $\widetilde X$. This shows that $L \to X$ is equipped with a canonical holomorphic Jacobi bundle structure.
\end{example}

\begin{definition}
The holomorphic Poisson manifold $(\widetilde X, \Pi, H)$ of Proposition \ref{prop:Poissonization} is called the \emph{Poissonization} of the holomorphic Jacobi manifold $(X, L, J)$.
\end{definition}

\subsubsection{From homogeneous holomorphic Poisson to holomorphic Jacobi manifolds}

\begin{proposition}\label{prop:hJ_hhP}
Every  homogeneous holomorphic Poisson manifold is the Poissonization of a canonical holomorphic Jacobi manifold around a non-singular point of the homogeneity vector field.
\end{proposition}

\begin{proof}
The proof is similar to that of Proposition \ref{prop:J_hP}. We report it here for completeness. So, let $(N, j_N)$ be a complex manifold, and let $(\Pi_N, H_N)$ be a homogeneous holomorphic Poisson structure on it. Denote by $\eta_N$ twice the real part of $H$ so that $H = \frac{1}{2} (\eta_N -ij_N \eta_N)$ and let $x \in N$ be a point such that $(H_N)_x \neq 0$. Then $\eta_N$ and $j_N \eta_N$ span a two dimensional distribution $D$ on a whole neighborhood $U$ of $x$. Since $H$ is holomorphic, the distribution $D$ is integrable. Let $M$ be the space of leaves of $D$. We assume, which is always possible, upon shrinking $U$ if necessary, that $M$ is a smooth manifold and the projection $p: U \to M$ is a surjective submersion. Clearly, $j_N$ descends to a complex structure $j_M$ on $M$. Let $X = (M, j_M)$. Upon shrinking $U$ further if necessary, we can always achieve the following situation: 1) there is a biholomorphism $F : U \to B \times X$, where $B \subset \C$ is a ball not containing $0$, and 2) $F$ intertwines $H_N$ and $w \frac{d}{dw}$, $w$ being the canonical complex coordinate on $B$. In the following, we use $F$ to identify $U$ with $B \times X$, and $H_N$ with $w \frac{d}{dw}$. Then, the projection $p : U = B \times X \to X$ is the projection onto the second factor. Holomorphic functions $f$ on $U$ that are homogeneous with respect to $H_N$, i.e.~such that $\Ll_{H_N} f = f$, are holomorphic functions that are linear in the \emph{nowhere vanishing} coordinate $w$, and they form an $\Oo_X$-module $\Ll$. The value of any such function on a fiber $\mathcal F$ of $p$ is completely determined by its value on a point of $\mathcal F$. Hence $\Ll$ is the module of holomorphic sections of a holomorphic line bundle $L \to X$. Since $\Pi_N$ is homogeneous with respect to $H_N$, the the corresponding Poisson bracket restricts to a Jacobi bracket $\{-,-\}$ on $L$. Finally, similarly as in the real case, $U$ can be embedded into the Poissonization of $(X, L, \{-,-\})$ as an open homogeneous holomorphic Poisson submanifold. 
\end{proof}

%\begin{remark}\label{rem:not_hJac}
%In the proof of Proposition \ref{prop:hJ_hhP} we can work with any pair $(\Pi_N, H_N)$ such that $\Pi_N$ is a (not necessarily Poisson) holomorphic bi-vector on $Y = (N, j_N)$ and $H_N$ is a holomorphic vector field on $Y$ such that $\Ll_{H_N} \Pi_N$. In this case, $\{-,-\}$ is still a well-defined holomorphic bi-derivation of $L \to X$, and it is a Jacobi bracket, i.e.~it satisfies the Jacobi identity, if and only if $\Pi$ is a holomorphic Poisson bi-vector.
%\end{remark}

\subsection{Holomorphic Jacobi, Jacobi-Nijenhuis and generalized contact manifolds}\label{sec:HJ_JN_GC}

As we briefly said, in the introduction, there are subtle differences between holomorphic Jacobi structures and holomorphic Poisson structures. One of these aspects is that  holomorphic line bundles are rank 2 (not rank 1) real vector bundles. Thus, a direct homogenization method won't work. To deal with this difficulty, the key step is to construct an intermediary real line bundle over a manifold $\widehat{M}$ starting from a holomorphic line bundle. We will see in this way that  holomorphic Jacobi structures are, in some sense, richer versions of holomorphic Poisson structures. 
\medskip 

Let $X = (M, j)$ be a complex manifold, let $L \to X$ be a holomorphic line bundle, and let $J \in \Gamma (D^{2,0}_\C L)$ be a holomorphic bi-derivation of $L$. Construct $(\widetilde X, \Pi, H)$ as in the proof of Proposition \ref{prop:Poissonization}, $\widetilde X = (\widetilde M, \widetilde j )$ (if, in particular, $(X,L,J)$ is a holomorphic Jacobi manifold, then $(\widetilde X, \Pi, H)$ is its Poissonization). We have $\Pi = \pi' + i \pi$ for some real Poisson bi-vectors $\pi, \pi'$ on $\widetilde M$, and $H = \frac{1}{2}(\eta -i\widetilde j  \eta)$, where $\eta$ is the real Euler vector field on $\widetilde M$. Additionally, $\Ll_\eta \pi = -\pi$, $\Ll_\eta \pi' = -\pi'$, and $\Ll_\eta \widetilde j = 0$. Now, notice that $\widetilde M$ can be given the structure of a principal $\R^\times$-bundle so that $\eta$ is the fundamental vector field corresponding to $1 \in \R$. Indeed, consider the real projective bundle $\R \P (L^\ast)$ (of the rank two real vector bundle $L \to M$). Denote it by $\widehat M$. Let $\ell \to \widehat M$ be the real dual of the real tautological line bundle over $\widehat M = \R \P (L^\ast)$. Clearly, $\widetilde M$ identifies canonically with the slit real dual of $\ell \to \widehat M$, i.e.~$\widetilde M = \ell^\ast \smallsetminus \{0\}$. Additionally, $\eta$ identifies with the (restriction of the) Euler vector field on (the total space of the real line bundle) $\ell^\ast \to \widehat M$. It follows that (see Remark \ref{rem:one-to-one})
\begin{enumerate}
\item $\pi,\pi'$ correspond to bi-derivations $\widehat J, \widehat J{}'$ of $\ell \to \widehat M$,
\item $\widetilde j$ corresponds to an endomorphism $\widehat j : D_\R \ell \to D_\R \ell$.
\end{enumerate}

\begin{theorem}\label{theor:hJ_JN_gc}
Let $X = (M,j)$ be a complex manifold, let $L \to X$ be a holomorphic line bundle, and let $J \in \Gamma(D^{2}_\C L)$ be a complex bi-derivation of the complex line bundle $L\to M$. Consider the bi-derivations $\widehat J, \widehat J{}'$ of the real line bundle $\ell \to \widehat M$, and the endomorphism $\widehat j : D_\R \ell \to D_\R \ell$, as constructed above. Then, the following conditions are equivalent
\begin{enumerate}
\item $(L, J)$ is a holomorphic Jacobi structure on $X$,
\item $(\widehat J, \widehat j)$ is a Jacobi-Nijenhuis structure on $\widehat M$, and $\widehat J{}' = \widehat J_{\widehat j}$,
\item $
\left(\begin{smallmatrix}
\widehat j & \widehat J{}^\sharp \\
0 & - \widehat j{}^\dag
\end{smallmatrix}\right)
$ is a generalized contact structure on $\ell \to \widehat M$, and $\widehat J{}' = \widehat J{}_{\widehat j}$.
\end{enumerate}
Additionally, $\widehat J$ (and hence $\widehat J{}'$) comes from a contact structure if and only if $J$ comes from a holomorphic contact structure.
\end{theorem}

\begin{proof}
The first part of the statement immediately follows from Proposition \ref{prop:hJ_hhP}, and Theorems \ref{theor:hP_PN_gc}, \ref{theor:JN_hPN}, \ref{theor:gc_hgc} (see also Remarks \ref{rem:one-to-one} and \ref{rem:one-to-one_2}). The last part can be easily proved in local coordinates (see also Example \ref{ex:hc} below).
\end{proof}

The situation is summarized in the following diagram

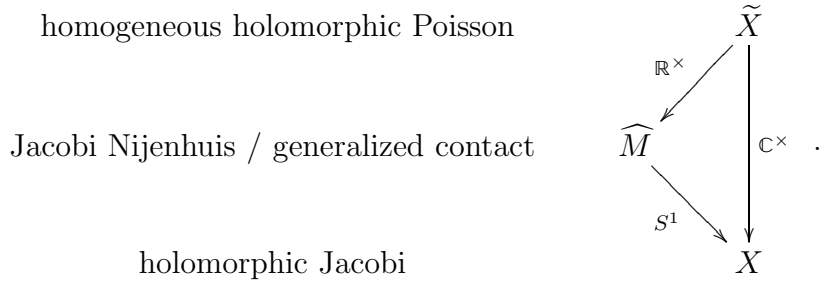
\[
\begin{array}{c}
\xymatrix{ \text{ homogeneous holomorphic Poisson} & & \widetilde X \ar[dl]_-{\text{$\R^\times$}} \ar[dd]^-{\text{$\C^\times$}}  \\
                  \text{Jacobi-Nijenhuis / generalized contact} & \widehat M \ar[dr]_-{\text{$S^1$}} & \\
                  \text{holomorphic Jacobi} & & X}
                  \end{array}.
\]

\begin{remark}
Let $(X, L, J)$ be a holomorphic Jacobi manifold, and let $\widehat J, \widehat j$ be as in Theorem \ref{theor:hJ_JN_gc}. From Proposition \ref{prop:JbH}, we deduce that $(\widehat{J}, \widehat{J}_{\widehat j})$ is a pair of compatible  Jacobi structures on $ \ell \to \widehat M$, i.e.  $\widehat{J} + \widehat{J}_{\widehat j}$ is also a Jacobi structure.
\end{remark}

\begin{remark}
Let $(X, L, J)$ be a holomorphic Jacobi manifold, and let $\eta, \pi, \widetilde j, \widehat J, \widehat j$ be as above. From $[\widetilde j  \eta, \eta] = 0$, it follows that $\widetilde j  \eta$ corresponds to a derivation $\Delta$ of $\ell \to \widehat M$. It is easy to see that $\Delta = \widehat j \mathbb 1$, where $\mathbb 1 : \Gamma (L) \to \Gamma(L)$ is the identity derivation. From 
\[
\Ll_{\widetilde j \eta} \pi_{\widetilde j} = \pi, \quad \text{and} \quad \Ll_{\widetilde j \eta} \pi = - \pi_{\widetilde j}
\]
(see (\ref{eq:L_jeta_pi})) now follows that
\[
[\widehat j \mathbb 1, \widehat J_{\widehat j}]^{SJ} =\widehat J \quad \text{and} \quad [ \widehat j \mathbb 1, \widehat J]^{SJ} = - \widehat J_{\widehat j},
\]
where $[-,-]^{SJ}$ is the \emph{Schouten-Jacobi bracket}.
\end{remark}

\begin{example}\label{ex:hc}
This example can be generalized to any holomorphic contact manifold and is originally due to Kobayashi \cite{K1959}. Let $X = \C^{2n+1}$, with complex coordinates 
\[
(t = r + is, z^k =x^k + iy^k, P_k = m_k + i q_k)
\]
and let 
\[
\theta = dt - P_k dz^k
\]
be the canonical holomorphic contact $1$-form on it. Then $\widetilde X = \C^{2n + 1} \times \C^\times$, and it has an additional nowhere vanishing coordinate $w = u + i v$. The holomorphic Euler vector field is $H = w \frac{\partial}{\partial w}$ and the homogeneous symplectic structure on $\widetilde X$ is 
\[
\Omega = dw \wedge (dt -P_k dz^k) - w dP_i \wedge dz^k.
\] 
The real Euler vector field is $\eta = u \frac{\partial}{\partial u} + v \frac{\partial}{\partial v}$. It is actually convenient to use a polar representation for the coordinate $w$ and write $w = \rho e^{i \varphi / 2}$. Then $\eta = \rho \frac{\partial}{\partial \rho}$ and $\varphi$ can be seen as the fiber coordinate in the $S^1$-bundle
\[
\R \P (\C^{2n+1} \times \C \to \C^{2n+1}) \cong \C^{2n+1} \times S^1 \to \C^{2n+1}.
\]
Theorem \ref{theor:hJ_JN_gc} now implies that there are two contact structures $\vartheta$ and $\vartheta_j$ on $\C^{2n+1} \times S^1$. They are locally given by
\[
\vartheta = \operatorname{cos} (\varphi/2) \vartheta_r -\operatorname{sin}(\varphi/2) \vartheta_s
\]
and
\[
\vartheta_j =  - \operatorname{sin}(\varphi/2) \vartheta_r - \operatorname{cos} (\varphi/2) \vartheta_s ,
\]
where
\[
\vartheta_r := dr-m_kdx^k + q_k dy^k, \quad \text{and} \quad \vartheta_s = ds - m_k dy^k -q_k dx^k.
\]
\end{example}
\section{Lie algebroids of a holomorphic Jacobi manifold}\label{Sec:LieAlgJacobi}

Recall that a Poisson bi-vector $\pi$ on a real manifold $M$ determines a Lie algebroid structure on $T^\ast M$, denoted by $(T^\ast M)_\pi$ and called the \emph{cotangent Lie algebroid of the Poisson manifold $(M, \pi)$}. The anchor of $(T^\ast M)_\pi$ is given by $\pi^\sharp : T^\ast M \to TM$ and the Lie bracket is $[-,-]_\pi$ (see (\ref{eq:[-,-]_pi})).
 
Similarly, a holomorphic Poisson structure $\Pi$ on a complex manifold $X = (M,j)$ determines a holomorphic Lie algebroid structure on $T^\ast X \to X$, denoted by $(T^\ast X)_\Pi$ and called the \emph{cotangent Lie algebroid of $(X, \Pi)$} \cite{LSX2008}. Specifically, the anchor $\rho : T^\ast X \to TX$ is given by 
\[
\rho (\omega) := 2 \operatorname{Re} \left( \Pi^\sharp (\omega - ij^\ast \omega) \right)
\]
for all $\omega \in \Omega^1 (X)$, and the bracket $[-,-]_\Pi$ acts on two holomorphic sections $\omega_1, \omega_2$ of $T^\ast X \to X$ as
\[
[\omega_1, \omega_2]_\Pi = \Ll_{\rho (\omega_1)} \omega_2 - \Ll_{\rho (\omega_2)} \omega_1 - d \langle \rho (\omega_1), \omega_2 \rangle.
\]

%\begin{remark}
%There is a diference between our conventions and conventions in \cite{LSX2008}. Namely, ``our'' cotangent Lie algebroid is $1/2$ of the cotangent Lie algebroid in \cite{LSX2008}, in the sense that the anchor and the bracket of the former are $1/2$ of the anchor and the bracket of the latter. The reason is that we decided to identify $T^\ast X$ and $(T^{1,0}M)^\ast$ via the isomorphism $\omega \mapsto \frac{1}{2} (\omega - ij^\ast \omega)$ (instead of $\omega \mapsto \omega - ij^\ast \omega$ as in \cite{LSX2008}). This will result in small differences between our formulations of results in \cite{LSX2008} and the original ones. We will not comment further on this point.
%\end{remark}

Let $\Pi = \pi_j + i \pi$, with $\pi, \pi_j$ real Poisson structures on $M$.

\begin{proposition}[Laurent-Gengoux, Sti\'enon, and Xu \cite{LSX2008}]\label{prop:real_cotangent}
The real (resp.~imaginary) Lie algebroid of the holomorphic Lie algebroid $(T^\ast X)_\Pi$ is $(T^\ast M)_{4 \pi_j}$ (resp.~$(T^\ast M)_{4 \pi}$).
\end{proposition}

We conclude the paper discussing the Jacobi analogue of Theorem \ref{prop:real_cotangent}. New features appear in the Jacobi setting due to the fact (hopefully already familiar to the reader) that, unlike for holomorphic Poisson manifolds, a holomorphic Jacobi manifold is not a real Jacobi manifold itself (but rather determines a real Jacobi structure on a certain circle bundle, see Section \ref{sec:HJ_JN_GC}).

First of all, recall that a Jacobi stucture $J$ on a real line bundle $L \to M$ does also determine a Lie algebroid structure on the first jet bundle $\mathfrak J^1_\R L$, denoted $(\mathfrak J^1_\R L)_J$ and called the \emph{jet algebroid of the Jacobi manifold $(M, L, J)$}. The anchor of $(\mathfrak J^1_\R L)_J$ is given by the composition of $J^\sharp : \mathfrak J_\R^1 L \to D_\R L$ followed by the symbol $\sigma : D_\R L \to TM$, and the Lie bracket is $[-,-]_J$ (see (\ref{eq:[-,-]_J})).
 
Similarly, a holomorphic Jacobi structure $(L, J)$ on a complex manifold $X = (M, j)$ determines a holomorphic Lie algebroid structure on $\mathfrak J^1 L \to X$, denoted by $(\mathfrak J^1 X)_J$ and called the \emph{jet Lie algebroid of $(X, L, J)$}. The anchor $\rho : \mathfrak J^1 L \to TX$ is given by 
\[
\rho (\theta) := 2 \operatorname{Re} \left( (\sigma \circ J^\sharp) (\theta) \right)
\]
for all $\theta \in \Gamma (\mathfrak J^1 L)$, and the bracket $[-,-]_J$ acts on two holomorphic sections $\theta_1, \theta_2$ of $\mathfrak J^1 L \to X$ as
\[
[\theta_1, \theta_2]_J = \Ll_{J^\sharp \theta_1} \theta_2 - \Ll_{J^\sharp \theta_2} \theta_1 - \mathfrak j^{1,0} J (\theta_1, \theta_2).
\]
Our next aim is to find the relationship between the real (resp.~the imaginary) Lie algebroid of the holomorphic Lie algebroid $(\mathfrak J^1 L)_J \to X$ and the Lie algebroid $(\mathfrak  J^1_\R L)_{\widehat J_j} \to \widehat M$ (resp.~$ (\mathfrak J^1_\R L)_{\widehat J} \to \widehat M$, see previous section for a definition of $\widehat{J}$). We start discussing the relationship between the vector bundles $\mathfrak J^1 L \to X$ and $\mathfrak  J^1_\R L \to \widehat M$.

Denote by 
\[
\widetilde{p}  : \widetilde X \longrightarrow X, \quad \widehat p  : \widehat M \longrightarrow X, \quad \text{and} \quad
q  : \widetilde X \longrightarrow \widehat M
\]
the projections. Additionally, for a complex manifold $X = (M,j)$, denote by $(T^\ast X)^{1,0}$ the $+i$-eigenbundle of $j^\ast : (T^\ast M)^\C \to (T^\ast M)^\C$, and by $\Omega^{1,0}(X)$ its sections.

\begin{proposition}\label{prop:lift} \
\begin{enumerate}
\item There are canonical isomorphisms of complex vector bundles
\[
(T^\ast \widetilde X)^{1,0} \cong \widetilde p{}^\ast \mathfrak J^1 L \cong  q^\ast \mathfrak J_\R^1 \ell, \quad \text{and} \quad \mathfrak J_\R^1 \ell \cong p^\ast \mathfrak J^1 L. 
\]
\item The isomorphism $(T^\ast \widetilde X)^{1,0} \cong \widetilde p{}^\ast \mathfrak J^1 L$ identifies pull-back sections with $1$-forms $\omega \in \Omega^{1,0}(\widetilde X)$ such that $\mathcal L_H \omega = \omega$ and $\mathcal L_{\overline H} \omega = 0$.
\item The isomorphism $(T^\ast \widetilde X)^{1,0} \cong  q^\ast \mathfrak J_\R ^1 \ell$ identifies pull-back sections with $1$-forms $\omega^{1,0}$ of the kind $\omega^{1,0} =  \omega -i\widetilde j ^\ast\omega $, with $\omega \in \Omega^1 (\widetilde X)$ a $1$-form homogeneous with respect to $\eta$, i.e.~$\mathcal L_\eta \omega = \omega$.
\item The isomorphism $\mathfrak J_\R^1 \ell \cong \widehat p{}^\ast \mathfrak J^1 L$ identifies pull-back sections with section $\theta$ of $\mathfrak J_\R ^1 \ell \to \widehat M$ such that $\mathcal L_{\widehat j \mathbb 1} \theta = \widehat j{}^\dag \theta$.
\end{enumerate}
\end{proposition}

\begin{proof}
The isomorphism $(T^\ast \widetilde X)^{1,0} \cong \widetilde p{}^\ast \mathfrak J^1 L$ identifies $\mathfrak j^{1,0} \lambda$ with $\partial \widetilde \lambda$, where $\lambda \in \Gamma (L)$, and claim (2) can be easily checked, e.g., in local coordinates.

The isomorphism $(T^\ast \widetilde X)^{1,0} \cong  q^\ast \mathfrak J_\R^1 \ell$ is obtained composing the isomorphism $(T^\ast \widetilde X)^{1,0} \to T^\ast \widetilde X$, $\omega \mapsto 2 \operatorname{Re} (\omega)$ with the isomorphism $T^\ast \widetilde M \cong p^\ast \mathfrak J_\R ^1 \ell$ in \cite[Proposition A.3.(3)]{V2015}. Claim (3) is then obvious from the explicit form of the isomorphism $T^\ast \widetilde M \cong p^\ast \mathfrak J_\R ^1 \ell$.

 For the isomorphism $\mathfrak J_\R^1 \ell \cong \widehat p{}^\ast \mathfrak J^1 L$, notice, first of all, that $\mathfrak J_\R ^1 \ell$ and $\mathfrak J^1 L$ have the same complex rank. Now, from claim (2), sections of $\mathfrak J^1 L$ are in one-to-one correspondence with $1$-forms $\omega^{1,0} \in \Omega^{1,0}(\widetilde X)$ such that
\begin{equation}\label{cond}
\mathcal L_H \omega^{1,0} = \omega^{1,0}, \quad  \text{and} \quad  \mathcal L_{\overline{H}} \omega^{1,0} = 0.
\end{equation}
 Let $\omega^{1,0} = \omega -i\widetilde j ^\ast \omega $
for some $\omega \in \Omega^1 (\widetilde M)$. Then conditions (\ref{cond}) are equivalent to $\mathcal L_\eta \omega = \omega$ and $\mathcal L_{\widetilde j  \eta} \omega = \widetilde j ^\ast \omega$. Hence, using claim (3), we find that sections of $\mathfrak J^1 L$ are in one-to-one correspondence with sections $\theta$ of $\mathfrak J_\R^1 \ell$ such that $\mathcal L_{\widehat j \mathbb 1} \theta = \widehat j^\dag \theta$. This correspondence is $C^\infty (M, \mathbb C)$-linear, so it comes from a morphism of vector bundles $\widehat p{}^\ast  \mathfrak J^1 L \to \mathfrak J_\R ^1 \ell$. It is easy to see that this morphism is injective, hence it is an isomorphism. This completes the proof of claim (1) and proves claim (4).
\end{proof}

%We denote by $\widehat P : \mathfrak J_\R ^1 \ell \cong p^\ast \mathfrak J^1 L \to \mathfrak J^1 L$ the natural projection. It covers projection $\widehat p : \widehat M \to X$. Moreover, we denote by $\psi \mapsto \theta_\psi$ the embedding $\Gamma (\mathfrak J^1 L) \to \Gamma (\mathfrak J_\R ^1 \ell)$ given by pulling-back along $\widehat p$.

The following theorem is the analogue of Proposition \ref{prop:real_cotangent} in the Jacobi setting. Additionally, it is the natural starting point for a discussion on the \emph{integration problem} for a holomorphic Jacobi manifold (see \cite{VW2019}).

\begin{theorem}\label{theor:real}
Let $(X, L, J)$ be a holomorphic Jacobi manifold, with $X = (M,j)$, and let $(\widehat J, \widehat j)$ be the associated Jacobi-Nijenhuis structure on the dual $\ell \to \widehat M$ of the tautological bundle over $\widehat M = \R \mathbb P (L^\ast)$. The real (resp.~imaginary) Lie algebroid of the holomorphic Lie algebroid $\mathfrak J^1 L$ acts naturally on the fibration $\widehat M \to M$, and $(\mathfrak J_\R^1 \ell)_{4\widehat J_j}$ (resp.~$(\mathfrak J_\R^1 \ell)_{4\widehat J}$) is the associated action Lie algebroid.
\end{theorem}

We refer to \cite[Section 3]{Kosmann-Mackenzie} for the notions of an (\emph{infinitesimal}) \emph{action of a Lie algebroid on a fibration} and the associated \emph{action Lie algebroid}. 

\begin{proof}[Proof of Theorem \ref{theor:real}]
First of all, the holomorphic Lie algebroid $(\mathfrak J^1 L)_J$ acts on the holomorphic line bundle $L \to X$ via the flat (holomorphic) connection $\nabla$ defined by
\begin{equation}\label{eq:flat_conn}
\nabla_\theta := (\iota \circ J^\sharp) (\theta),
\end{equation}
for all $\theta \in \Gamma (\mathfrak J^1 L)$, where $\iota : D^{1,0} L \to DL$ is the canonical isomorphism. Since the symbol of $\nabla_\theta$ is precisely $\rho (\theta)$, then $\nabla$ is well-defined. To check that it is flat it is enough to check that $[\nabla_{\theta_1}, \nabla_{\theta_2}]$ and $\nabla_{[\theta_1, \theta_2]_J}$ agree on holomorphic sections of $L$, whenever $\theta_1, \theta_2$ are of the form $\mathfrak j^1 \lambda_1, \mathfrak j^1 \lambda_2$ for some holomorphic section $\lambda_1, \lambda_2$ of $L$. This easily follows from the Jacobi identity for the Jacobi bracket $\{-,-\}$ and the formula
\[
[\mathfrak j^{1,0} \lambda_1, \mathfrak j^{1,0} \lambda_2] = \mathfrak j^{1,0} \{ \lambda_1, \lambda_2 \}.
\]
Now, $(\mathfrak J^1 L)_J$ acts on the complex dual line bundle $L^\ast \to X$ as well via the dual connection. When a holomorphic Lie algebroid acts on a holomorphic vector bundle $E$, then its real Lie algebroid acts on $E$ as well and the action is uniquely defined by the same formula on holomorphic sections \cite{LSX2008}. Hence the real Lie algebroid of $(\mathfrak J^1 L)_J$ acts on $L^\ast \to X$. Denote by $\nabla$ again the flat $(\mathfrak J^1 L)_J$-connection in $L^\ast$.

For every $\theta \in \Gamma (\mathfrak J^1 L)$, $\nabla_\theta$ is a section of the holomorphic gauge algebroid of $L^\ast$. In particular, it is a (real) derivation of the (real) vector bundle $L^\ast \to X$. Accordingly, it corresponds to a linear vector field on the total space $L^\ast$ of $L^\ast \to X$. It is easy to see that linear vector fields on $L^\ast$ descend to $\widehat p$-projectable vector fields on $\widehat M = \R \P (L^\ast)$. Hence the real Lie algebroid of $(\mathfrak J^1 L)_J$ acts on the manifold $\widehat M$, and the pull-back $\widehat p{}^\ast \mathfrak J^1 L = \mathfrak J^1_\R \ell \to \widehat M$ is equipped with the action Lie algebroid structure. We claim that the latter coincides with $(\mathfrak J^1_\R \ell)_{4 \widehat J_j}$. To see this it is enough to check that the structure maps of the two Lie algebroids agree on pull-back sections of $\widehat p{}^\ast \mathfrak J^1 L \to \widehat M$. Even more, it is enough to check that the anchors and the brackets agree on sections of the form $\widehat p{}^\ast \mathfrak j^{1,0} \lambda$, with $\lambda$ a holomorphic sections of $L \to X$. To do this we pass through $\widetilde X$. Namely, $\widetilde p^\ast \mathfrak j^{1,0} \lambda$ identifies with $\partial \widetilde \lambda$ through the isomorphism $\widetilde p{}^\ast \mathfrak J^1 L \cong (T^{1,0}M)^\ast$ in Proposition \ref{prop:lift}, and we have $\Ll_H \widetilde \lambda = \widetilde \lambda$. It follows that $f_\lambda := \operatorname{Re} (\widetilde \lambda)$ is a homogeneous function with respect to the principal $\R^\times$-bundle $\widetilde X \to \widehat M$. Hence, $f_\lambda$ corresponds to a section $\widehat \lambda$ of $\ell \to \widetilde M$, and $\widehat p{}^\ast \mathfrak j^{1,0} \lambda$ corresponds to $\mathfrak j^1_\R \widehat \lambda$ via the isomorphism $\widehat p^\ast \mathfrak J^1 L \cong \mathfrak J^1_\R \ell$ (see Proposition \ref{prop:lift} again). Now, let $\lambda_1, \lambda_2$ be holomorphic sections of $L \to X$ and compute
\[
[\mathfrak j^1_\R \widehat \lambda_2 , \mathfrak j^1_\R \widehat \lambda_2 ]_{4\widehat J_j} = \mathfrak j^1_\R \{ \widehat \lambda_1 , \widehat \lambda_2\}_{4\widehat J_j}
\]
where $\{-,-\}_{4\widehat J_j}$ is the Jacobi bracket corresponding to the Jacobi structure $4 \widehat J_j$. To complete with the brackets, it is enough to check that 
\[
\{ \widehat \lambda_1 , \widehat \lambda_2\}_{4\widehat J_j} = \widehat {\{ \lambda_1, \lambda_2\}}.
\]
This is easy, indeed $\{ \lambda_1, \lambda_2\}$ corresponds to the homogeneous function $\{ \widetilde \lambda_1, \widetilde \lambda_2 \}_\Pi$ on $\widetilde X$ with respect to $H$. Now from \cite[Corollary 2.4]{LSX2008}
\[
 \operatorname{Re}(\{ \widetilde \lambda_1, \widetilde \lambda_2 \}_\Pi )=  \{ \operatorname{Re}(\widetilde \lambda_1), \operatorname{Re}(\widetilde \lambda_2 )\}_{4\pi_j}
\]
which is the homogeneous function with respect to $\eta$ corresponding to $\{ \widehat \lambda_1 , \widehat \lambda_2\}_{4\widehat J_j} $, as claimed.

It remains to check that the anchor $\rho$ of $(\mathfrak J^1_\R \ell)_{4\widehat J_j}$ and the anchor $\rho'$ of the action Lie algebroid $\widehat p{}^\ast  \mathfrak J^1 L \to \widehat M$ agree on sections of the form $\widehat p{}^\ast \mathfrak j^{1,0} \lambda$, where $\lambda$ is a holomorphic section of $L \to X$. Using the same notations as above, compute
\begin{equation}\label{eq:rho}
\rho (\mathfrak j^1_\R \widehat \lambda) (f) = \sigma (\{ \widehat \lambda , -\}_{4 \widehat J_j}) (f).
\end{equation}
for all (local) functions $f$ on $\widehat M$. The right hand side of (\ref{eq:rho}) is the push forward of the function 
\[
\{ \operatorname{Re}( \widetilde \lambda), f \}_{4\pi_j} 
\]
on $\widetilde X$, the latter being constant along the fibers of $q : \widetilde X \to \widehat M$. 
In fact, it is enough to choose $f$ to be the (pushforward of the) quotient $g_1/g_2$ of two fiber-wise $\R$-linear functions $g_1,g_2$ on $L^\ast$. We can even choose $g_1,g_2$ to be the real parts of two fiber-wise $\C$-linear, holomorphic functions $\gamma_1, \gamma_2$ on $L^\ast$. So $\gamma_1, \gamma_2$ are actually holomorphic sections of $L \to X$. In this case, from \cite[Corollary 2.4]{LSX2008} again,
\[
\{ \operatorname{Re} (\widetilde \lambda), f \}_{4\pi_j} = \frac{\operatorname{Re} (\{ \widetilde \lambda, \widetilde \gamma_1\}_\Pi )g_2 - \operatorname{Re} (\{ \widetilde \lambda, \widetilde \gamma_2\}_\Pi) g_1 }{g_2^2}.
\]
On the other hand
\[
\rho '(\mathfrak j^1_\R \widehat \lambda) (f) = \frac{\operatorname{Re} \left(\iota \{  \lambda, -\} (\gamma_1)\right) g_2 - \operatorname{Re} \left(\iota \{  \lambda, -\} (\gamma_2)\right) g_1 }{g_2^2},
\]

and, to conclude, it is enough to notice that, from $\lambda, \gamma$ being holomorphic, it follows that $\iota\{\lambda, - \} (\gamma) = \{ \lambda, \gamma \}$ corresponds to the homogeneous function $\{ \widetilde \lambda, \widetilde \gamma \}_\Pi$ on $\widetilde X$, $\gamma = \gamma_1, \gamma_2$.  

As for the Lie algebroid $(\mathfrak J^1_\R \ell)_{\widehat J}$, it is enough to notice that, if $J \in \Gamma (D^{2,0} L)$ is a holomorphic Jacobi structure on $L$, then $iJ$ is also a holomorphic Jacobi structure and $\widehat{iJ} = \widehat J_j$. We leave details to the reader. We only remark that the action of the imaginary Lie algebroid of $(\mathfrak J^1 L)_J$ on $\widehat M \to X$ is induced by a linear action on $L$, given by a flat connection $\nabla'$ defined by 
\[
\nabla'_\theta := (\iota \circ J^\sharp) (i \theta)= (\iota \circ J^\sharp) (j_{\mathfrak J^1 L} \theta),
\]
for all $\theta \in \Gamma (\mathfrak J^1 L)$.
\end{proof}

\begin{remark}
	New tools such as the construction of the vector bundle 	$\mathfrak J_\R^1 L  \to \widehat M$ were needed because holomorphic Jacobi structures  are much more complicated than  holomorphic Poisson ones  \cite{LSX2008}. In particular, this  construction allows us to unravel the
relationship between the real  (resp. imaginary Lie algebroid of the holomorphic Lie algebroid $(\mathfrak J^1 L)_J \to X$ and the Lie algebroid $(\mathfrak  J^1_\R L)_{\widehat J} \to \widehat M$ (resp.~$ (\mathfrak J^1_\R L)_{\widehat J_j} \to \widehat M$)).  Clearly,  as already outlined, unlike  in the holomorphic Poisson case, $(\mathfrak  J^1_\R L)_{\widehat J} \to \widehat M$ cannot be just the same as the real Lie algebroid  of $(\mathfrak J^1 L)_J \to X$ since these two have different base manifolds. This shows once again that the holomorphic Jacobi case  is richer than the holomorphic Poisson case. 

\end{remark}
\section*{Acknowledgements}
The authors thank Ping Xu for useful discussions. Luca Vitagliano  thanks Ping Xu and Aissa Wade for inviting him to the Mathematics Department of PSU in February 2015, when part of this project has been developed. Luca Vitagliano's visit to PSU was supported by the Shapiro Visitor Program at Penn State. Luca Vitagliano  is member of the GNSAGA of INdAM, Italy.

\end{document}